\documentclass[12pt]{amsart}

\usepackage[english]{babel}
\usepackage[toc,page]{appendix}
\usepackage{amsmath}
\usepackage{amsthm}
\usepackage{amssymb}
\usepackage{mathrsfs}
\usepackage{enumerate}
\usepackage{bookmark}
\usepackage[notcite, final, notref]{showkeys}
\usepackage{dsfont}
\usepackage[dvips]{color}
\newcommand{\suchthat}{\;\ifnum\currentgrouptype=16 \middle\fi|\;}

\newtheorem{theorem}{Theorem}[section]
\newtheorem{problem}[theorem]{Problem}

\newtheorem{proposition}[theorem]{Proposition}
\newtheorem{corollary}[theorem]{Corollary}
\newtheorem{definition}[theorem]{Definition}

\theoremstyle{definition}
\newtheorem{remark}[theorem]{Remark}

\allowdisplaybreaks

\newcommand{\R}{\mathcal{R}}

\usepackage{tikz}
\usepackage[siunitx]{circuitikz}

\usepackage{xcolor}

\title[Functions of Fueter variables]{A general setting for functions of Fueter variables: differentiability, rational functions, Fock module and related topics}

\author[D. Alpay]{Daniel Alpay}
\address{(DA) 
Faculty of Mathematics, Physics, and Computation\\
Schmid College of Science and Technology\\
Chapman University\\
One University Drive\\
Orange, California 92866\\
USA}
\email{alpay@chapman.edu}

\author[I. L. Paiva]{Ismael L. Paiva}
\address{(ILP)
Schmid College of Science and Technology\\
Chapman University\\
One University Drive\\
Orange, California 92866\\
USA}
\email{depaiva@chapman.edu}

\author[D. C. Struppa]{Daniele C. Struppa}
\address{(DCS) 
Faculty of Mathematics, Physics, and Computation\\
Schmid College of Science and Technology\\
Chapman University\\
One University Drive\\
Orange, California 92866\\
USA}
\email{struppa@chapman.edu}

\begin{document}

\begin{abstract}
We develop some aspects of the theory of hyperholomorphic functions whose values are taken in a Banach algebra over a field -- assumed to be the real or the complex numbers -- and which contains the field. Notably, we consider Fueter expansions, Gleason's problem, the theory of hyperholomorphic rational functions, modules of Fueter series, and related problems. Such a framework includes many familiar algebras as particular cases. The quaternions, the split quaternions, the Clifford algebras, the ternary algebra, and the Grassmann algebra are a few examples of them.
\end{abstract}

\maketitle


\date{today}
\tableofcontents
\section{Introduction}
\setcounter{equation}{0}
In \cite{MR1509533}, Fueter studied quaternionic-valued functions and introduced non-commuting hypercomplex variables which allow power series expansion of hypercomplex functions, meaning functions that belong to the kernel of the operator $D$ defined in \eqref{dirac} in the quaternionic setting. Such variables are now known as Fueter variables. In the present work, we extend their definition to functions in the kernel of $D$ and whose values are taken in a Banach algebra $\mathcal{A}$ over a field $\mathbb{K}$ and which contains $\mathbb{K}$. We assume $\mathbb{K}$ to be either $\mathbb{R}$ or $\mathbb{C}$.\smallskip

After defining the Fueter variables in Section \ref{sec-gen-prin}, we start studying the notion of derivative of $\mathcal{A}$-valued functions, which is followed by Fueter expansions and Gleason's problem, the theory of rational functions, and spaces of Fueter series, which include the Drury-Arveson space and the Fock space. Such problems have already been considered in some particular cases -- for instance, when $\mathcal{A}$ is the quaternionic setting, the Clifford algebra, the setting of split quaternions, or the real ternary algebra \cite{MR2275397, alss_IJM, MR2124899, alpay2018gleason}.\smallskip

The central object of our study are functions
\begin{equation}
f:\mathcal{A} \rightarrow \mathcal{A}.
\label{function}
\end{equation}
To set the notation, we let $\mathfrak{a}_1,\ldots, \mathfrak{a}_n\in\mathcal{A}\setminus\mathbb{K}$ be the generators of $\mathcal{A}$ different from the identity element of $\mathbb{K}$, i.e., $\mathfrak{a}_0=1$. The number of generators $n$ is assumed to be finite, but it could, in fact, be taken to be infinity for all computations presented here. The only aspect that would require a deeper understanding, in this case, is the Cauchy-Fueter operator, defined in \eqref{dirac}, which would become a differential operator of infinitely many variables.\smallskip

In some algebras, all ``directions'' are given by the elements $\mathfrak{a}_k$, $k=0,1,\ldots,n$. In those algebras, there exist coefficients $c_{jkl}\in\mathbb{K}$, $j,k,l\in\{0,1,\hdots,n\}$ such that
\begin{equation}
\label{relation}
\mathfrak{a}_j \mathfrak{a}_k = \sum_{l=0}^n c_{jkl} \mathfrak{a}_l
\end{equation}
for every $j,k\in\{1,\ldots,n\}$. An example of such algebras is the quaternionic setting. Denoting the complex units of the quaternions by $i$, $j$ and $k$, the product of two elements, say $ij$ is $k$.\smallskip

In other algebras, on the other hand, the product $\mathfrak{a}_j\mathfrak{a}_k$ might result in a new ``direction''. An example of it is the real ternary algebra, which is generated by the number $1$ and an element $e$ which is not real (nor complex). However, $e^2$ gives a new direction in such an algebra. The quaternions themselves can also be seen as an example of such algebras when the elements $1$, $i$, and $j$ are considered their generators. In this case, $ij$ results in a new direction of the algebra. This shows the quaternions explicitly as an example of a Clifford algebra.\smallskip

Then, in order to include in our study arbitrary algebras $\mathcal{A}$ where \eqref{relation} need not hold, we introduce a set of $t$-uples $\mathfrak{I}$ which are associated with all linearly independent directions of the algebra. If for a certain $j$ and $k$ the product $\mathfrak{a}_j\mathfrak{a}_k$ gives a new direction, then $(j,k)$ might be an element of $\mathfrak{I}$ and $\mathfrak{a}_{(j,k)}\equiv \mathfrak{a}_j\mathfrak{a}_k$. Note that, in some algebras, there might be multiple ways to build $\mathfrak{I}$. For instance, if $\mathfrak{a}_j\mathfrak{a}_k=-\mathfrak{a}_k\mathfrak{a}_j$, as is the case of the Clifford algebra, and the Grassmann algebra, either $(j,k)$ or $(k,j)$ should be part of $\mathfrak{I}$, not both, since they are not associated with linearly independent directions. In this case, if $(j,k)\in\mathfrak{I}$, $\mathfrak{a}_k\mathfrak{a}_j=-\mathfrak{a}_{(j,k)}$. Moreover, the direction $e_0=1$ is not included in $\mathfrak{I}$. \smallskip

Instead of using the set $\mathfrak{I}$ directly, we refer to a map from it into $\mathbb{Z}_m=\{1,2,\cdots,m\}$, where $m$ is the cardinality of $\mathfrak{I}$. So instead of considering indexes that take value in $\mathfrak{I}$, they are taken in $\mathbb{Z}_m$. Moreover, to avoid confusion, we denote the ``independent directions'' of $\mathcal{A}$ by $e_k$ instead of $\mathfrak{a}_k$. For instance, if we say the quaternions are generated by $\mathfrak{a}_0=1$, $\mathfrak{a}_1=i$, and $\mathfrak{a}_2=j$, then in our new notation: $e_0=\mathfrak{a}_0$, $e_1=\mathfrak{a}_1$, $e_2 = \mathfrak{a}_2$, and $e_3 = \mathfrak{a}_1 \mathfrak{a}_2$. \smallskip

With the above discussion, the algebra $\mathcal{A}$ can be identified with the space $\mathbb{K}^{m+1}$ in the following way
\begin{equation}
\mathbb{K}^{m+1} \simeq \mathcal{A} = \left\{a=\sum_{k=0}^m a_k e_k\suchthat a_0,a_1,\hdots,a_m\in\mathbb{K}\right\}.
\label{isomorp}
\end{equation}
In particular, expression \eqref{relation} can be rewritten in this general context as
\begin{equation}
e_j e_k = \sum_{\ell=0}^m (\chi_\ell)_{jk} e_\ell,
\label{prod-rule}
\end{equation}
where the matrices $\chi_\ell$ belong to $\mathbb{K}^{(m+1)\times(m+1)}$ for every $\ell\in\{0,1,\hdots,m\}$. We call the matrices $\chi_\ell$ the {\it characteristic operators of the algebra $\mathcal{A}$} since they encode all properties of the algebra product. One can also see each $\chi_\ell$ as a metric tensor-like object associated with the algebra $\ell$-th direction of $\mathcal{A}$. In fact, observe that for any $a,b\in\mathcal{A}$, if we write
\[
ab = \sum_{\ell=0}^m c_\ell e_\ell,
\]
equation \eqref{prod-rule} leads to
\begin{equation}
c_\ell = \sum_{j,k=1}^m a_j (\chi_\ell)_{jk} b_k \equiv a\chi_\ell b
\label{chi-def}
\end{equation}
for every $\ell\in\{0,1,\hdots,m\}$. However, note that the operators $\chi_\ell$ are not represented by symmetric matrices in every algebra $\mathcal{A}$. Also, note that, in the above expression, the identification between elements of $\mathcal{A}$ with points in $\mathbb{K}^{m+1}$ -- expressed by \eqref{isomorp} -- is being used. \smallskip

Also, we endow the algebra $\mathcal{A}$ with an involution, denoted by $\dag$, with the following properties:
\begin{itemize}
\item $\forall a,b\in\mathcal{A}, \ (ab)^\dag=b^\dag a^\dag$;
\item $\forall k \in\mathbb{K}\subset\mathcal{A}, \ kk^\dag = k^\dag k = |k|^2$.
\end{itemize}
We note that, in general, $aa^\dag$ is not a real (nor a complex) number and, then, it is not always the case that $aa^\dag=a^\dag a$ holds.\smallskip

Because $\mathcal A$ is a Banach algebra, we also assume it has a norm $N$ for which
\begin{equation}
\label{Nab}
N(ab)\le N(a)N(b),\quad \forall a,b\in\mathcal{A}.
\end{equation}
Such a norm may or may not be induced by the involution $\dag$. However, we assume $\ N(a)=N(a^\dag)$ for every $a\in\mathcal{A}$.\smallskip

Note that because $\mathbb{K}\subset\mathcal{A}$, $N$ is also a norm in $\mathbb{K}$. Therefore, without loss of generality, we can assume $N(1)=1$. Then, we can write $N(k)=|k|$, where $|k|$ denotes the usual norm of $k$ in $\mathbb{C}$. Moreover, observe that
\begin{equation}
N(ka) = |k| N(a)
\label{basic-norm-prop}
\end{equation}
for every $a\in\mathcal{A}$.

Introduced some definitions in the algebra $\mathcal{A}$, we look back at the functions $f$ of the type \eqref{function} and observe that, because of \eqref{isomorp}, they can be identified with functions
\begin{equation}
f: \mathbb{K}^{m+1}\rightarrow\mathcal{A}.
\label{function2}
\end{equation}
In our study, we consider in particular the subset of such functions which are $\mathbb{K}$-analytic -- i.e., real analytic functions if $\mathbb{K}=\mathbb{R}$ or complex holomorphic functions if $\mathbb{K}=\mathbb{C}$.\smallskip

Our first goal is to study the analyticity of such functions. In the case of functions $f$ of a complex variable $z$, we say $f$ is holomorphic at a certain point $a\in\mathbb{C}$ if the following notion of a derivative holds:
\[
\frac{df}{dz}(a)=\lim_{z\rightarrow a}\frac{f(z)-f(a)}{z-a}\in\mathbb{C}.
\]
This is the case if and only if the Cauchy-Riemann equations are satisfied at $a$.\smallskip

However, a straightforward extension of the definition of holomorphicity to functions that take values in other settings, i.e., the definition in terms of what is generally called a Fr\'{e}chet derivative does not always lead to interesting analytical structures.\smallskip

The classic example of where such an extension ``fails'' is the quaternionic setting. In fact, if $f$ was a function of a quaternionic variable $q$ and the condition for $f$ to be hyperholomorphic at a certain point $a\in\mathbb{H}$ was
\[
\lim_{q\rightarrow a}\left(f(q)-f(a)\right)(q-a)^{-1}\in\mathbb{H}
\]
or
\[
\lim_{q\rightarrow a}(q-a)^{-1}\left(f(q)-f(a)\right)\in\mathbb{H},
\]
then one would conclude that $f$ is a linear function of $q$. Because of it, one considers different definitions to build analysis tools upon, e.g., slice hyperholomorphicity. \smallskip

Nonetheless, Malonek showed in \cite{malonek1990new} that, for real Clifford algebras with $n$ non-real generators, the idea of Fr\'{e}chet derivatives could be used as a ``good'' definition of hyperholomorphicity. However, instead of seeing the function as a function of a single variable in the algebra or, equivalently, a function of $n+1$ real variables, the Fr\'{e}chet derivative was taken by considering the function as a function of $n$ Fueter variables. With this result, Malonek studied the so-called Fueter series in \cite{malonek1990power}, i.e., power series of Fueter variables.\smallskip

We now describe the content of the paper. In Section \ref{sec-hyper}, we generalize Malonek's results on hyperholomorphicity to the algebras $\mathcal{A}$ we already described. In Section \ref{sec-poly}, we define Fueter polynomials, which form the building blocks for Fueter power series expansions. Then, in Section \ref{sec-series}, we study the convergence of such series centered at a point where the function is hyperholomorphic. We also introduce in this section the Cauchy product at the center of the power series, i.e., the convolution of coefficients of power series. Such a definition is necessary since, due to the lack of commutativity, the pointwise product of two hyperholomorphic functions is not necessarily hyperholomorphic. We emphasize that, because the Cauchy product is defined at the center of the power series, it is center dependent. After introducing the convolution, we discuss Gleason's problem. In Section \ref{sec-rat}, we define and characterize hyperholomorphic rational functions, which are important tools in analysis. Moreover, we introduce reproducing kernel Banach spaces and study multipliers in those spaces in Section \ref{sec-banach}. As an example of such spaces, in Section \ref{sec-arveson}, we present the counterpart of the Drury-Arveson space in our setting and, in Section \ref{sec-blaschke}, study Blaschke factors in such a space. As another example, we also introduce the analogous of the Fock space in Section \ref{sec-fock}.\smallskip

Before that, we start by defining and presenting other motivations for the Fueter variables in Section \ref{sec-gen-prin}.

\section{A general principle}
\setcounter{equation}{0}
\label{sec-gen-prin}

Consider $\mathcal{A}$-valued functions $f$ of $m+1$ $\mathbb{K}$-valued variables $v_k$, $k\in\{0,1,\cdots,m\}$, which are of class $C^1$ in an open subset $\Omega$ of $\mathbb K^{m+1}$, and
satisfying
\begin{equation}
Df=0.
\label{d-kernel}
\end{equation}
In the above expression, $D$ denotes the Cauchy-Fueter or Dirac operator, which is defined as
\begin{equation}
D=D_0+\sum_{k=1}^m e_k D_k,
\label{dirac}
\end{equation}
where
\[
D_j=\frac{\partial}{\partial v_j}, \qquad 
\]
for every $j\in\{0,1,\cdots,m\}$. Such functions $f$ are called left $D$-hyperholomorphic, left $\mathcal{A}$-analytic \cite[Definition 3.1 p. 119]{MR2129645}, or left monogenic \cite{malonek1990new}. The word left is used because one can also consider functions $f$ that belongs to the kernel of $D$ when it acts on $f$ from the right. Here, because we focus on functions that satisfy \eqref{d-kernel}, we omit the word left and just call them $D$-hyperholomorphic, $\mathcal{A}$-analytic, or monogenic. The results we obtain can be easily adapted for right $D$-hyperholomorphic functions.\smallskip

Then, denoting $v=(v_0,v_1,\cdots,v_m)\in\mathbb{K}^{m+1}$, we have
\[
\begin{split}
\frac{{\rm d}f}{\rm dt}(tv)
&=\sum_{k=0}^m v_k \ \frac{\partial f}{\partial v_k}(tv)\\
&=\sum_{k=1}^m v_k \ \frac{\partial f}{\partial v_k}(tv)-v_0\left(\sum_{k=1}^m e_\alpha \ \frac{\partial f}{\partial v_k}(tv)\right)\\
&=\sum_{k=1}^m(v_k-e_k v_0) \ \frac{\partial f}{\partial v_k}(tv)
\end{split}
\]
and so
\[
f(v)-f(0)=\int_0^1\frac{\rm d}{\rm dt}f(tv)dt=\sum_{k=1}^m(v_k-e_k v_0)\int_0^1\frac{\partial f}{\partial v_k}(tv)dt.
\]

\begin{definition}
We call the functions
\begin{equation}
\zeta_k(v)=v_k -e_k v_0,\quad k\in\mathbb{Z}_m,
\end{equation}
the Fueter variables associated with $D$. They are also called total regular variables by Delanghe \cite{delanghe1972singularities}.
\end{definition}

\begin{remark}
{\rm In the Clifford algebra setting, one usually considers functions $f$ of $n+1$ variables in $\mathbb{K}$ that belong to the kernel of the Cauchy-Fueter operator
\[
D = D_0 + \sum_{k=1}^n D_k e_k.
\]
We notice that this represents particular cases of the framework presented here since those functions $f$ are also in the kernel of the Cauchy-Fueter operator $D$ given by expression \eqref{dirac}.}
\end{remark}

More generally, for a fixed $w\in\Omega$,
\[
\begin{split}
\frac{{\rm d}f}{\rm dt}(tv+(1-t)w)
&=\sum_{k=0}^m (v_k-w_k) \ \frac{\partial f}{\partial v_k}(tv+(1-t)w)\\
&=\sum_{k=1}^m(\zeta_k(v)-\zeta_k(w)) \ \frac{\partial f}{\partial v_k}(tv+(1-t)w)
\end{split}
\]
holds true for every $v\in\Omega$. Therefore,
\begin{equation}
\label{gleason!!!}
f(v)-f(w)=\int_0^1\frac{\rm d}{\rm dt}f(tv)dt=\sum_{k=1}^m(\zeta_k(v)-\zeta_k(w))\int_0^1\frac{\partial f}{\partial v_k}(tv+(1-t)w)dt.
\end{equation}

Following Malonek's approach for the Clifford algebra in \cite{malonek1990new}, consider now $\zeta=(\zeta_1,\zeta_2,\cdots,\zeta_m)$ and let $\mathscr{H}^m$ be the set of all such vectors. Then, there is a one-to-one correspondence between $\mathbb{K}^{m+1}$ and $\mathscr{H}^m$:
\[
\mathbb{K}^{m+1} \simeq \mathscr{H}^m = \left\{\zeta=(\zeta_1,\cdots,\zeta_m)\suchthat \zeta_k=v_k-e_k v_0; v_0,v_k\in\mathbb{K}\right\}.
\]
Because of this correspondence between those spaces, we use the notations $f(v)$ and $f(\zeta)$ indistinguishably. Moreover, we often write $\xi=\zeta(w)\in\mathscr{H}^m$.

\begin{remark}
\label{not-submodule}
\label{remark-non-hyper}
{\rm Although it is clear that $\mathscr{H}^m\subset\mathcal{A}^m$, observe that $\mathscr{H}^m$ is not an $\mathcal{A}$-submodule. In fact, $c\zeta$ belongs to $\mathscr{H}^m$ for an arbitrary $\zeta\in\mathscr{H}^m$ if and only if $c\in\mathbb{K}$.\smallskip

As a consequence, we have that the product of two $D$-hyperholomorphic functions need not be $D$-hyperholomorphic. For instance, if $f(v)=\zeta_j \zeta_k$ where $j,k\in\mathbb{Z}_m$, we have
\[
\zeta_j\zeta_k= v_j v_k-e_j v_0v_k-e_k v_0v_j +e_j e_k v_0^2
\]
and, then,
\begin{equation}
D(\zeta_j\zeta_k)=\left[e_j,e_k\right] v_0^2,
\label{non-hyperhol}
\end{equation}
where $\left[e_j,e_k\right]\equiv e_j e_k - e_k e_j$ is the commutator of $e_j$ and $e_k$. Expression \eqref{non-hyperhol} does not vanish in general since $e_j$ and $e_k$ might not commute. However, it follows from \eqref{non-hyperhol} that $D(\zeta_j\zeta_k+\zeta_k\zeta_j)=0$. More generally, as we will see in the next section, symmetrized products of the Fueter variables (and, in particular, powers of a single  variable) are $D$-hyperholomorphic. Examples of algebras where expression \eqref{non-hyperhol} is different from zero include the quaternions, the split quaternions and the Grassmann algebra  -- the latter will be studied in more details in a future work.}
\end{remark}

We will denote by
\begin{equation}
(\R_k(w)f)(v) \equiv \int_0^1 \frac{\partial f}{\partial v_k}(tv+(1-t)w) dt,
\label{bs-general}
\end{equation}
where $w\in\Omega$ is a fixed element, the backward-shift operator centered at $w$. Then, equation \eqref{gleason!!!} becomes
\begin{equation}
f(v)-f(w)=\sum_{k=1}^m (\zeta_k(v)-\zeta_k(w))(\R_k(w) f)(v).
\label{general-sol}
\end{equation}
Moreover, we note that
\[
(\R_k(w)f)(w) = \frac{\partial f}{\partial v_k}(w).
\]

\begin{remark}
{\rm Sometimes, we denote the backward-shift operator $(\R_k(w)f)(v)$ defined in \eqref{bs-general} by $(\R_k(\xi)f)(\zeta)$.}
\end{remark}

If a restriction to $C^2(\Omega)$ functions is made, then $\R_k(w) f(v)$ is $D$-hyperholomorphic. Moreover, if $f$ belongs to $C^p(\Omega)$ for some $p\in\mathbb{N}$, the process given by \eqref{non-hyperhol} can be iterated $p$ times, generating the so-called {\it Fueter polynomials}, which is the object of study of Section \ref{sec-poly}. In particular, $C^\infty(\Omega)$ functions give origin to {\it Fueter series}. Even then, we note that each of the individual pointwise products $(\zeta_k(v)-\zeta_k(w))(\R_k(w) f)(v)$ in \eqref{general-sol} do not need to be $D$-hyperholomorphic. Their sum, however, which is equal to $f(v)-f(w)$, is $D$-hyperholomorphic.

\section{Hyperholomorphicity of functions from $\mathbb{K}^{m+1}$ into $\mathcal{A}$}
\setcounter{equation}{0}
\label{sec-hyper}

In this section, we show how Malonek's work in \cite{malonek1990new} generalizes from the Clifford algebra to a more general scenario. The calculations follow the arguments in that paper. \smallskip

We start by endowing the algebra $\mathcal{A}^m$ with the Hermitian form
\begin{equation}
\label{form}
\left[\zeta,\xi\right]_\mathcal{A}=\sum_{k=1}^m \xi_k^\dagger \zeta_k,
\end{equation}
which in general is $\mathcal{A}$-valued. \smallskip

Defining
\[
\mathfrak{h}_k = \left(0,\cdots,0,1,0\cdots,0\right)\in\mathscr{H}^m \quad k\in\mathbb{Z}_m
\]
and
\[
\mathfrak{h}_0 = -\left(e_1,\cdots,e_m\right)\in\mathscr{H}^m,
\]
it follows that an arbitrary element $\zeta\in\mathscr{H}^m$ can be written as
\begin{equation}
\zeta = \sum_{k=0}^mv_k\mathfrak{h}_k.
\label{decomposition}
\end{equation}
Moreover,
\begin{equation}
\left[\zeta,\mathfrak{h}_k\right]_\mathcal{A}=\zeta_k
\label{component}
\end{equation}
for every $k\in\mathbb{Z}_m$ and
\[
\left[\zeta,\mathfrak{h}_0\right]_\mathcal{A} = -v_0\sum_{k=1}^m e_k^2 - \sum_{k=1}^m v_k e_k.
\]

\begin{remark}
{\rm Note that if $u\in\mathcal{A}^m$, then there exist coefficients $u_k\in\mathcal{A}$, $k\in\mathbb{Z}_m$ such that
\[
u = \sum_{k=1}^m u_k \mathfrak{h}_k.
\]
In this sense, the set $\{\mathfrak{h}_k\}_{k\in\mathbb{Z}_m}$ is a canonical basis for $\mathcal{A}^m$.}
\end{remark}

Because of the above remark, instead of using the induced norm
\[
N\left(\left[\zeta,\zeta\right]_\mathcal{A}\right)^{1/2} = N\left(\sum_{k=1}^m\zeta_k^\dag\zeta_k\right)^{1/2}
\]
in $\mathcal{A}^m$, we define
\begin{equation}
\left\Vert\zeta\right\Vert_{\mathcal{A}^m} = \left(\sum_{k=1}^m N(\left[\zeta,\mathfrak{h}_k\right]_\mathcal{A})^2\right)^{1/2}.
\label{norm-am}
\end{equation}

\begin{remark}
{\rm Expression \eqref{decomposition} can only be written because of the embedding of $\mathscr{H}^m$ in $\mathcal{A}^m$. As a matter of fact, observe that even though \eqref{decomposition} holds, $\zeta_k\mathfrak{h}_k\not\in\mathscr{H}^m$ in general. This is a different way to see the discussion in Remark \ref{remark-non-hyper}.}
\end{remark}

Denote by $\mathscr{L}(\mathcal{A}^m,\mathcal{A})$ the set of all $\mathcal{A}$-linear operators from $\mathcal{A}^m$ into $\mathcal{A}$. We recall that $L\in\mathscr{L}$ is said to be $\mathcal{A}$-linear from the left if
\[
L(au+bv)=aL(u)+bL(v)
\]
for every $a,b\in\mathcal{A}$ and $u,v\in\mathcal{A}^m$. Because \eqref{decomposition} holds, it follows that $L\in\mathscr{L}(\mathscr{H}^m,\mathcal{A})$ is $\mathcal{A}$-linear from the left if
\[
L(a\zeta+b\xi)=aL(\zeta)+bL(\xi)
\]
and, moreover, there exist constants $A_k\in\mathcal{A}$, $k\in\mathbb{Z}_m$, such that $L$ is uniquely characterized by
\begin{equation}
L(\zeta)= \zeta_1 A_1 + \cdots + \zeta_m A_m.
\label{linear-map}
\end{equation}

\begin{remark}
{\rm A similar construction can be made for operators that are $\mathcal{A}$-linear from the right.}
\end{remark}

Next, we present the definition of a differentiable function from an open subset of $\mathscr{H}^m$ into $\mathcal{A}$.

\begin{definition}
Let $f$ be a continuous map from an open set $\Omega\in\mathscr{H}^m$ into $\mathcal{A}$. Then, $f$ is said to be left hyperholomorphic if it has a Fr\'{e}chet derivative, i.e., if there exists an $\mathcal{A}$-linear from the left linear map $L\in\mathscr{L}(\mathscr{H}^m,\mathcal{A})$ such that
\begin{equation}
\lim_{\Delta\zeta\rightarrow0} \frac{N\left(f(\xi+\Delta\zeta)-f(\xi)-L(\Delta\zeta)\right)}{\left\Vert\Delta\zeta\right\Vert_{\mathcal{A}^m}} = 0
\label{der}
\end{equation}
for every $\xi\in\Omega$.
\end{definition}

\begin{proposition}
If $f$ is left hyperholomorphic, then the linear map $L$ in \eqref{der} is unique.
\end{proposition}

\begin{proof}
Equations \eqref{linear-map} and \eqref{der} imply that
\[
f(\xi+\Delta\zeta)-f(\xi) = \Delta\zeta_1 A_1 + \cdots + \Delta\zeta_m A_m + o(\left\Vert\Delta\zeta\right\Vert_{\mathcal{A}^m}),
\]
where
\[
\lim_{\Delta\zeta\rightarrow0} \frac{o(\left\Vert\Delta\zeta\right\Vert_{\mathcal{A}^m})}{\left\Vert\Delta\zeta\right\Vert_{\mathcal{A}^m}}=0.
\]
\end{proof}

\begin{theorem}
Let $f$ be a $\mathbb{K}$-analytic function. Then, $f$ is left hyperholomorphic if and only if it is (left) $D$-hyperholomorphic.
\end{theorem}

\begin{proof}
The fact that $f$ is $\mathbb{K}$-analytic leads to
\begin{equation}
f(\zeta+\Delta\zeta)-f(\zeta) = \Delta f(\zeta) =  \Delta v_0 \frac{\partial f}{\partial v_0} + \Delta v_1 \frac{\partial f}{\partial v_1} + \cdots + \Delta v_m \frac{\partial f}{\partial v_m} + o(|\Delta v|)
\label{k-hol}
\end{equation}
with
\[
\lim_{\Delta v\rightarrow0} \frac{o(|\Delta v|)}{|\Delta v|}=0.
\]
Note that the order of the products in the right-hand side of \eqref{k-hol} does not matter here. Now, consider the variable substitution $v_0=\zeta_0$ and
\[
v_k = \zeta_0 e_k + \zeta_k \quad k\in\mathbb{Z}_m.
\]
Then, equation \eqref{k-hol} implies that
\begin{eqnarray*}
\Delta f(\zeta) & = & \Delta\zeta_0 \left(D_0f + \sum_{k=1}^m e_k D_kf\right) + \sum_{k=1}^m \Delta\zeta_k D_kf + o(\left\Vert\Delta\zeta\right\Vert_{\mathcal{A}^m}) \\
                & = & \Delta\zeta_0 Df + \left[\Delta\zeta,\nabla_v\right]f + o(\left\Vert\Delta\zeta\right\Vert_{\mathcal{A}^m}),
\end{eqnarray*}
where $\nabla_v\equiv\sum_{k=1}^m D_k \mathfrak{h}_k$. Comparing the above expression with \eqref{linear-map}, we conclude that $f$ is hyperholomorphic if and only if $Df=0$, i.e., $f$ is $D$-hyperholomorphic.
\end{proof}

\begin{corollary}
The differential of $f$ is given by
\[
df = d\zeta_0 Df + \left[d\zeta,\nabla_v\right]f.
\]
\end{corollary}

\begin{corollary}
If $f$ is $D$-hyperholomorphic from the left and from the right, then
\[
df = \sum_{k=1}^m \frac{\partial f}{\partial v_k} \times d\zeta_k,
\]
where $\times$ denotes the symmetrized product, which is defined next -- see equation \eqref{symmetrized-product}.
\end{corollary}

\section{Fueter polynomials}
\setcounter{equation}{0}
\label{sec-poly}

Generalizing classical cases, the Fueter variables generate polynomials which are hyperholomophic. They are called Fueter polynomials.\smallskip

We recall that in a non-commutative algebra $\mathcal{A}$, the {\it symmetrized product} of $a_1,\ldots, a_N\in \mathcal{A}$ is defined by
\begin{equation}
a_1\times a_2\times\cdots\times a_N=\frac{1}{N!}\sum_{\sigma\in S_N}a_{\sigma(1)}a_{\sigma(2)}\cdots a_{\sigma(N)},
\label{symmetrized-product}
\end{equation}
where the sum is over the set $S_N$ of all permutations on $N$ indexes. It is worth mentioning that such a product is used in quantum mechanics for fermionic systems. Also, observe that, in general, the product $\times$ is non-associative, i.e.,
\[
(a\times b)\times c\neq a\times(b\times c).
\]
Moreover, if $\mathcal{A}$ is commutative, the symmetrized product reduces to the regular product. Furthermore, as pointed out by Malonek in \cite{malonek1990power}, for every $k\in\mathbb{N}$
\begin{equation}
(a_1+\cdots+a_N)^k = \sum_{\substack{\alpha\in\mathbb{N}_0^N;\\ |\alpha|=k}} \frac{k!}{\alpha!} a^\alpha,
\end{equation}
where $a^\alpha=a_1^{\alpha_1}\times\cdots\times a_N^{\alpha_N}$.

\begin{proposition}
Let $\alpha=(\alpha_1,\cdots, \alpha_m)\in\mathbb N^m$. The symmetrized product of the terms $\zeta_k^{\alpha_k}$, $k\in\mathbb{Z}_m$, simply denoted by
$\zeta^{\alpha}$, is hyperholomorphic.
\end{proposition}

\begin{proof}
For this proof, we use the same method that appears in \cite{MR0265618,MR1509533}, which was also used in \cite{alss_IJM} in the setting of split quaternions. We start rewriting $\zeta^{\alpha}$ as
\[
\zeta^{\alpha}=\sum_{\sigma\in S_{|\alpha|}}\sum_{u=1}^m\sum_{\substack{k\in\mathbb{Z}_{|\alpha|}; \\ \sigma(k)=u}}
\zeta_{\sigma(1)}\cdots\zeta_{\sigma(k-1)}\zeta_u\zeta_{\sigma(k+1)}\cdots \zeta_{\sigma(|\alpha|)},
\]
where $|\alpha|$ denotes the sum of the components of $\alpha$ and $\mathbb{Z}_{|\alpha|}=\{1,2,\hdots,|\alpha|\}$. Therefore,
\[
\begin{split}
D(\zeta^{\alpha})&=\sum_{\sigma\in S_{|\alpha|}}\sum_{u=1}^m\sum_{\substack{k\in\mathbb{Z}_{|\alpha|}; \\ \sigma(k)=u}}
e_u\zeta_{\sigma(1)}\cdots\zeta_{\sigma(k-1)}\zeta_{\sigma(k+1)}\cdots \zeta_{\sigma(|\alpha|)}-\\
&\hspace{5mm}-\sum_{\sigma\in S_{|\alpha|}} \sum_{u=1}^m \sum_{\substack{k\in\mathbb{Z}_{|\alpha|}; \\ \sigma(k)=u}}
\zeta_{\sigma(1)}\cdots\zeta_{\sigma(k-1)}e_u\zeta_{\sigma(k+1)}\cdots \zeta_{\sigma(|\alpha|)}.
\end{split}
\]

Hence,
\[
\begin{split}
v_0D(\zeta^{\alpha})&=\sum_{\sigma\in S_{|\alpha|}}\sum_{u=1}^m\sum_{\substack{k\in\mathbb{Z}_{|\alpha|}; \\ \sigma(k)=u}}
v_0e_u\zeta_{\sigma(1)}\cdots\zeta_{\sigma(k-1)}\zeta_{\sigma(k+1)}\cdots \zeta_{\sigma(|\alpha|)}-\\
&\hspace{5mm}-\sum_{\sigma\in S_{|\alpha|}}\sum_{u=1}^m\sum_{\substack{k\in\mathbb{Z}_{|\alpha|}; \\ \sigma(k)=u}}
\zeta_{\sigma(1)}\cdots\zeta_{\sigma(k-1)}e_uv_0\zeta_{\sigma(k+1)}\cdots \zeta_{\sigma(|\alpha|)}\\
&=\sum_{\sigma\in S_{|\alpha|}}\sum_{u=1}^m\sum_{\substack{k\in\mathbb{Z}_{|\alpha|}; \\ \sigma(k)=u}}
(v_0e_u-v_u+v_u)\zeta_{\sigma(1)}\cdots\zeta_{\sigma(k-1)}\zeta_{\sigma(k+1)}\cdots \zeta_{\sigma(|\alpha|)}-\\
&\hspace{5mm}-\sum_{\sigma\in S_{|\alpha|}}\sum_{u=1}^m\sum_{\substack{k\in\mathbb{Z}_{|\alpha|}; \\ \sigma(k)=u}}
\zeta_{\sigma(1)}\cdots\zeta_{\sigma(k-1)}(e_uv_0-v_u+v_u)\zeta_{\sigma(k+1)}\cdots \zeta_{\sigma(|\alpha|)}\\
&=-\sum_{\sigma\in S_{|\alpha|}}\sum_{u=1}^m\sum_{\substack{k\in\mathbb{Z}_{|\alpha|}; \\ \sigma(k)=u}}
\zeta_u\zeta_{\sigma(1)}\cdots\zeta_{\sigma(k-1)}\zeta_{\sigma(k+1)}\cdots \zeta_{\sigma(|\alpha|)}-\\
&\hspace{5mm}+\sum_{\sigma\in S_{|\alpha|}}\sum_{u=1}^m\sum_{\substack{k\in\mathbb{Z}_{|\alpha|}; \\ \sigma(k)=u}}
\zeta_{\sigma(1)}\cdots\zeta_{\sigma(k-1)}\zeta_u\zeta_{\sigma(k+1)}\cdots \zeta_{\sigma(|\alpha|)}\\
&=0
\end{split}
\]
since the sum is made on all the permutations and
\[
\begin{split}
\sum_{\sigma\in S_{|\alpha|}}\sum_{u=1}^m\sum_{\substack{k\in\mathbb{Z}_{|\alpha|}; \\ \sigma(k)=u}}
v_u\zeta_{\sigma(1)}\cdots\zeta_{\sigma(k-1)}\zeta_{\sigma(k+1)}\cdots \zeta_{\sigma(|\alpha|)}=\\
&\hspace{-50mm}=\sum_{\sigma\in S_{|\alpha|}}\sum_{u=1}^m\sum_{\substack{k\in\mathbb{Z}_{|\alpha|}; \\ \sigma(k)=u}}
\zeta_{\sigma(1)}\cdots\zeta_{\sigma(k-1)}v_u\zeta_{\sigma(k+1)}\cdots \zeta_{\sigma(|\alpha|)}.
\end{split}
\]
\end{proof}

With the above proposition, our next result on Fueter polynomials follows trivially.

\begin{corollary}
Any Fueter polynomial is hyperholomorphic. In particular, $(\zeta-\xi)^\alpha$ is hyperholomorphic for every $\alpha\in\mathbb{N}_0^m$ and constant $\xi\in\mathscr{H}^m$.
\end{corollary}

With the discussion presented in this section so far, it is clear that symmetrized products of Fueter variables are fundamental for the construction of a hyperholomorphic Fueter polynomial. Now, we note that such a product also appear naturally in an expansion of $\mathbb{K}$-analytic functions by repeatedly iterating the backward-shift operator. As an example, observe that applying \eqref{general-sol} to $(\R_k(\xi)f)(\zeta)$ gives
\[
\begin{split}
f(\zeta) &= f(\xi) + \sum_{k=1}^m (\zeta_k-\xi_k) \frac{\partial f}{\partial v_k}(\xi) \\
&\hspace{5mm} + \sum_{k=1}^m \sum_{1\le j \le k} \left[(\zeta_k-\xi_k)(\zeta_j-\xi_j)+(\zeta_j-\xi_j)(\zeta_k-\xi_k)\right] (\R_k(\xi)\R_j(\xi)f)(\zeta),
\end{split}
\]
where $f$ is assumed to be of class $C^2(\Omega)$. Note that $\R_k(\xi)\R_j(\xi)=\R_j(\xi)\R_k(\xi)$ for every $j,k\in\mathbb{Z}_m$.\smallskip

More generally, let $\R(\xi)^\alpha \equiv \R_1(\xi)^{\alpha_1}\R_2(\xi)^{\alpha_2}\cdots \R_m(\xi)^{\alpha_m}$. Then,
\begin{eqnarray*}
(\R^\alpha(\xi)f)(\xi) & = & \frac{\partial^{|\alpha|} f}{\partial v^\alpha}(\xi)\int_0^1 \int_0^1 \cdots\int_0^1 t_1^{|\alpha|-1} t_2^{|\alpha|-2} \cdots t_{|\alpha|-1} \ dt_1 \ dt_2 \cdots dt_{|\alpha|}  \\
             & = & \frac{1}{\alpha!} \frac{\partial^2 f}{\partial v_k v_j}(\xi),
\end{eqnarray*}
where $\partial v^\alpha = \partial v_1^{\alpha_1} \partial v_2^{\alpha_2}\cdots \partial v_m^{\alpha_m}$. As a consequence, we can write for every function $f$ of class $C^p(\Omega)$
\begin{equation}
f(\zeta) = f(\xi) + \sum_{\substack{\alpha\in\mathbb{N}_0^m; \\ |\alpha|< p}} (\zeta-\xi)^\alpha \frac{\partial^{|\alpha|} f}{\partial v^\alpha}(\xi) + \sum_{\substack{\alpha\in\mathbb{N}_0^m; \\ |\alpha|=p}} (\zeta-\xi)^\alpha (\R(\xi)^\alpha f)(\zeta).
\label{iteration}
\end{equation}

\section{Fueter series and Gleason's problem}
\setcounter{equation}{0}
\label{sec-series}

At this stage, we restrict our study to hyperholomorphic functions, which were defined in Section \ref{sec-hyper}, i.e., functions which are $\mathbb K$-analytic and belong to the kernel of the Cauchy-Fueter operator. We note that in some settings every hyperholomorphic function is automatically $\mathbb{K}$-analytic -- this is true, for instance, for elliptic systems.\smallskip

Our concern now is with the result of successive iterations of $\R(\xi)$ to a function $f$, as started in \eqref{iteration}. Such a process leads, at least formally, to the Fueter series
\begin{equation}
f(\zeta)=\sum_{\alpha\in\mathbb N_0^{m}}(\zeta-\xi)^{\alpha} f_\alpha(\xi).
\label{fueter123}
\end{equation}
Our question is whether such a series converges in an open neighborhood of $\xi$. \smallskip

In order to set the notation for the theorem that answers such a question, recall that if a function $f$ that takes value in $\mathbb{K}^{m+1}$ is $\mathbb{K}$-analytic in a neighborhood $\Omega(w)$ of a point $w$, then the power series
\begin{equation}
f(v)=\sum_{\alpha\in\mathbb N_0^{m+1}}(v-w)^\alpha f_\alpha(w),
\label{ps-k}
\end{equation}
where
\begin{equation}
f_\alpha(w)=\frac{1}{\alpha!}\frac{\partial^{|\alpha|}f}{\partial v^\alpha}(w),
\label{ps-coeff}
\end{equation}
converges for every $v\in\Omega(w)$, i.e., there exist $K>0$ and strictly positive numbers $r_1,\cdots, r_m$ such that
\begin{equation}
N(f_\alpha)\le\frac{K}{r_0^{\alpha_0}\cdots r_m^{\alpha_m}}\equiv\frac{K}{r^\alpha}.
\label{bound-falpha}
\end{equation}

Also, let $\sigma_1,\cdots,\sigma_m$ be strictly positive numbers such that $\sigma_k< r_k$, $k\in\mathbb{Z}_m$ and recall that we are denoting by $\xi$ the element in $\mathscr{H}^m$ that corresponds to the point $w\in\mathbb{K}^{m+1}$.\smallskip

Following \cite{malonek1990power}, we define
\[
\mathscr{U}_{\xi}(\sigma) = \left\{\zeta\in\mathscr{H}^m \suchthat N(\zeta_k-\xi_k)\le \sigma_k, k\in\mathbb{Z}_m\right\}\subset\mathscr{H}^m.
\]

A last notation remark before the presentation of the theorem, we note that whenever $\alpha$ takes value in $\mathbb{N}_0^m$, the coefficient $f_\alpha$ refers to $f_{(0,\alpha)}$, where $(0,\alpha)\in\mathbb{N}_0^{m+1}$. With that set, we state a theorem on the convergence of Fueter series.

\begin{theorem}
Let $f$ given by \eqref{ps-k} be $D$-hyperholomorphic in a neighborhood of $w\in\mathbb{K}^{m+1}$. Then, the Fueter series given by \eqref{fueter123} converges absolutely for every $\zeta\in\mathscr{U}_\xi(\sigma)$.
\end{theorem}

\begin{proof}
We start by observing that \eqref{fueter123} and
\[
f(\zeta)=\sum_{k=0}^\infty\sum_{\substack{\alpha\in\mathbb N_0^m;\\ |\alpha|=k}}(\zeta-\xi)^{\alpha} f_\alpha(\xi),
\]
at least formally, are both ``natural'' representations of power series with Fueter variables. On the one hand, however, they generally have different domains of convergence. Such a remark is important, for instance, if $\mathcal{A}$ is the Clifford algebra \cite{brackx1982clifford}. On the other hand, if one considers the domain $\mathscr{U}_\xi(\sigma)$, the two expressions, if they converge, coincide because
\begin{equation}
N((\zeta-\xi)^\alpha)\le\prod_{k=1}^m N(\zeta_k-\xi_k)^{\alpha_k}\le \sigma^\alpha < r^\alpha
\label{quai-de-la-rapee}
\end{equation}
and, as a consequence,
\[
N\left(\sum_{\alpha\in\mathbb N_0^{m}}(\zeta-\xi)^{\alpha} f_\alpha(\xi)\right) \le \sum_{\alpha\in\mathbb N_0^{m}}\sigma^{\alpha} N\left(f_\alpha(\xi)\right) \le K \sum_{\alpha\in\mathbb N_0^{m}}\left(\frac{\sigma}{r}\right)^{\alpha}<\infty,
\]
where we have used expression \eqref{bound-falpha}. Hence, \eqref{fueter123} converges absolutely in the the domain $\mathscr{U}_\xi(\sigma)$.
\end{proof}

Before presenting the next result, we just define elements $\iota_k\in\mathbb{N}_0^m$ whose $j$-th entries are characterized by
\[
\left(\iota_k\right)_j = \delta_{jk}, \quad j,k\in\mathbb{Z}_m.
\]

\begin{proposition}
Let $f$ be given by \eqref{fueter123}. Moreover, let $\R_k(w)$ be given by \eqref{bs-general} and $\zeta(w)=\xi$. Then,
\[
\R_k(w)f(v) = \sum_{\substack{\alpha\in\mathbb{N}_0^m;\\ \alpha\geq \iota_k}} \frac{\alpha_k}{|\alpha|} (\zeta(v)-\xi)^{\alpha-\iota_k} f_\alpha.
\]
\label{rk-ps}
\end{proposition}

\begin{proof}
The proof is a generalization of the one presented for the quaternions in \cite{MR2124899}. First, note that
\[
\frac{\partial(\zeta-\xi)^{\alpha}}{\partial v_k}(v) = \alpha_k (\zeta(v)-\xi)^{\alpha-\iota_k}.
\]
Therefore,
\begin{eqnarray*}
\R_k(\xi)(\zeta-\xi)^\alpha & = & \int_0^1 \frac{\partial(\zeta-\xi)^{\alpha}}{\partial v_k}(tv+(1-t)w) dt \\
             & = & \int_0^1 \alpha_k (\zeta(tv+(1-t)w)-\xi)^{\alpha-\iota_k} dt \\
             & = & \int_0^1 \alpha_k t^{|\alpha|-1} (\zeta-\xi)^{\alpha-\iota_k}(v) dt \\
             & = & \frac{\alpha_k}{|\alpha|} (\zeta-\xi)^{\alpha-\iota_k}
\end{eqnarray*}
since $\zeta(tv+(1-t)w)=t\zeta(v)+(1-t)\xi$.

The above proves the proposition and justifies the name backward-shift operator given to $\R_k(\xi)$.
\end{proof}

\begin{definition}
We define for $\alpha,\beta\in\mathbb N_0^m$ and $u,v$ in the algebra $\mathcal{A}$, the Cauchy (or convolution) product at $\xi\in\mathscr{H}^m$ by
\begin{equation}
(\zeta-\xi)^{\alpha}u\odot_{\scriptscriptstyle \xi} (\zeta-\xi)^{\beta}v=(\zeta-\xi)^{\alpha+\beta}uv
\end{equation}
\end{definition}

\begin{remark}
{\rm In the quaternionic setting, the Cauchy product at the origin can be defined using the Cauchy-Kowaleskaya theorem (see \cite{MR618518}), and is then called Cauchy-Kowalweskay product. It can also be applied to functions that are not necessarily hyperholomorphic at the origin. The principle of the Cauchy-Kowaleskaya product is the following. The real components of a quaternionic valued function $f$ such that $D_{CF}f=0$ satisfy a set of linear partial differential equations to which the Cauchy-Kovalesvkaya theorem is applicable. The solution to this system is uniquely determined by the initial condition $f(0,x_1,x_2,x_3)$, and the Cauchy-Kovaleskaya product of $f$ and $g$ is defined by the {\sl pointwise} product of the initial conditions $f(0,x_1,x_2,x_3)g(0,x_1,x_2,x_3)$. In the case of Fueter series, the Cauchy-Kovaleskaya is, in fact, not needed.}
\label{remarkbastille}
\end{remark}

Now, we extend the product $\odot_{\scriptscriptstyle \xi}$ and the $\odot_{\scriptscriptstyle \xi}$-inverse for power series. Remembering that $\xi=\zeta(w)$, let
\[
f(\zeta)=\sum_{\alpha\in\mathbb N^m}(\zeta-\xi)^\alpha f_\alpha \quad \text{and} \quad g(\zeta)=\sum_{\alpha\in\mathbb{N}_0^m}(\zeta-\xi)^\alpha g_\alpha.
\]
Setting $v_0=w_0$, we obtain power series in $v_k-w_k$, $k\in\mathbb{Z}_m$, with coefficients in the algebra $\mathcal{A}$, i.e.,
\[
f(w_0,v_1,\ldots, v_m)=\sum_{\alpha\in\mathbb N_0^m}(v-w)^\alpha f_\alpha \quad \text{and} \quad g(w_0,v_1,\ldots, v_m)=\sum_{\alpha\in\mathbb N_0^m}(v-w)^\alpha g_\alpha.
\]
Observe that now the coefficients $f_\alpha$ (and $g_\alpha$) commute with the variables. We consider, then, the product $f(w_0,v_1,\ldots,v_m)g(w_0,v_1,\ldots,v_m)$ and write it as
\[
f(w_0,v_1,\ldots,v_m)g(w_0,v_1,\ldots,v_m)=\sum_{\alpha\in\mathbb N_0^m}(v-w)^\alpha h_\alpha,
\]
where
\begin{equation}
h_\alpha\equiv\sum_{\substack{\gamma\in\mathbb{N}_0^m;\\ \alpha-\gamma\ge0}}f_{\alpha-\gamma}g_\gamma.
\label{coef-conv}
\end{equation}
With that, we define
\[
(f\odot_{\scriptscriptstyle \xi} g)(\zeta)\equiv\sum_{\alpha\in\mathbb N_0^m}(\zeta-\xi)^\alpha h_\alpha.
\]

Now, assume $f_0\in\mathcal{A}$ is invertible. Then, in a neighborhood of $w$ in $\Omega\subset\mathbb{K}^m$, we have
\[
(f(w_0,v_1,\ldots, v_n))^{-1}=\sum_{\alpha\in\mathbb N_0^n}(v-w)^\alpha d_\alpha
\]
for some coefficients $d_\alpha\in\mathcal{A}$. We set
\begin{equation}
f^{-\odot_{\scriptscriptstyle \xi}}(\zeta)=\sum_{\alpha\in\mathbb N_0^m}(\zeta-\xi)^\alpha d_\alpha.
\label{odot-inv}
\end{equation}
This is well defined since the coefficients $f_\alpha$ (and hence $d_\alpha$) are uniquely determined by $f$.\smallskip

An important fact is that the Cauchy product in the way we defined is dependent on the center of the power series (or of the polynomials). As an example, consider the quaternionic setting, where $e_1=i$, $e_2=j$, and $e_3=k$. The polynomial $P(\zeta)=\zeta_1^2 j$ can be seen as $P=p_1\odot_{\scriptscriptstyle 0} p_2$, where $p_1(\zeta)=\zeta_1 k$ and $p_2(\zeta)=\zeta_1 i$. However, $p_1$ and $p_2$ can also be rewritten with a center at $\xi$ such that $\xi_k=e_k$, $k\in\mathbb{Z}_3$. In fact,
\[
p_1(\zeta) = (\zeta_1-i)k - j \quad \text{and} \quad p_2(\zeta) = (\zeta_1-i)i-1.
\]
The convolution of $p_1$ and $p_2$ at this center is, then,
\[
Q(\zeta) = (p_1\odot_{\scriptscriptstyle \xi} p_2)(\zeta) = (\zeta_1-i)^2 j+j = \zeta_1^2 j -2\zeta_1 k
\]
since $\zeta_1$ and $i$ commute. Hence, $P\neq Q$, i.e., the convolution of $p_1$ and $p_2$ at the origin is different from their convolution at $\xi$.\smallskip

We also observe that the right-hand side Cauchy product can be defined as
\[
(f\odot_{\scriptscriptstyle \xi}^R g)(z)\equiv\sum_{\alpha\in\mathbb N_0^m} h_\alpha (\zeta-\xi)^\alpha,
\]
where
\[
f(\zeta)=\sum_{\alpha\in\mathbb N^m}f_\alpha(\zeta-\xi)^\alpha, \quad g(\zeta)=\sum_{\alpha\in\mathbb{N}_0^m}g_\alpha(\zeta-\xi)^\alpha
\]
and the coefficients $h_\alpha$ are once again given by \eqref{coef-conv}.\smallskip

We now introduce Gleason's problem.

\begin{problem}
Given a hyperholomorphic function $f$ with domain given by $\mathscr{U}_\xi(\sigma)$, find functions $g_1,\cdots,g_m$ such that
\[
f(\zeta)-f(\xi) = \sum_{k=1}^m (\zeta_k-\xi_k)\odot_{\scriptscriptstyle \xi} g_k(\zeta)
\]
for every $\zeta\in\mathscr{U}_\xi(\sigma)$.
\label{gleason}
\end{problem}

Observe that expression \eqref{general-sol} is a solution to this problem with the pointwise product -- instead of the $\odot_{\scriptscriptstyle \xi}$-product. The disadvantage of the poinwise product is that, as already discussed, it is not necessary hyperholomorphic. However, \eqref{general-sol} is usefull for us here, since it shows that $g_k=\R_k(\xi)f$, $k\in\mathbb{Z}_m$ is a solution to Problem \ref{gleason}. In fact, for every $\zeta$ such that $v_0=w_0$, the $\odot_{\scriptscriptstyle \xi}$-product coincides with the pointwise product and Gleason's problem \eqref{gleason} becomes equivalent to find solutions of \eqref{general-sol}. In turn, this fact allows writing in general
\[
f(\zeta)-f(\xi) = \sum_{k=1}^m (\zeta_k-\xi_k)\odot_{\scriptscriptstyle \xi} (\R_k(\xi)f)(\zeta).
\]

Now, we show that those are not all the solutions to the problem. Let $\mathcal{G}$ denote the space of functions $f\in\mathcal{G}$ for which exist $g_1,g_2,\cdots,g_m\in\mathcal{G}$ that solve Gleason's problem. The space $\mathcal{G}$ is said to be resolvent-invariant. Moreover, let $\mathscr{R}$ be the space of $\R_k(\xi)$-invariant functions, called backward-shift-invariant, i.e., the space for which $g_k=\R_k(\xi)f$. Our next results aim to prove that $\mathscr{R}\subsetneqq\mathcal{G}$. We allow, in general, the functions to be matrix-valued. First, we characterize the elements of $\mathcal{G}$ with the following proposition.

\begin{proposition}
A function $f$ belongs to a finite-dimensional resolvent-invariant space $\mathcal{G}$ if and only if it can be spanned by the columns of a matrix-valued function of the type
\begin{equation}
G(\zeta) = G(\xi)\odot_{\scriptscriptstyle \xi}\left(I-\sum_{k=1}^m(\zeta_k-\xi_k)A_k\right)^{-\odot_{\scriptscriptstyle \xi}},
\label{g-span}
\end{equation}
where $A_k$, $k\in\mathbb{Z}_m$, are constant matrices with entries in $\mathcal{A}$.
\label{resolvent}
\end{proposition}

\begin{proof}
Let $G$ be a matrix-valued hyperholomorphic function whose columns form a basis of $\mathcal{G}$ and $f\in\mathcal{G}$. Then, by definition, there exist a constant column matrix $\eta$ with entries in $\mathcal{A}$ such that $f=G\eta$ and functions $g_1,\hdots,g_m\in\mathcal{G}$ such that
\[
f(\zeta)-f(\xi) = \sum_{k=1}^m (\zeta_k-\xi_k) \odot_{\scriptscriptstyle \xi} g_k(\zeta).
\]
Moreover, there exist constant matrices $A_k$, $k\in\mathbb{Z}_m$, such that $g_k=GA_k\eta$. Hence, the above expression can be rewritten as
\[
\left[G(\zeta) - G(\xi)\right]\eta = G(\zeta)\odot_{\scriptscriptstyle \xi}\sum_{k=1}^m (\zeta_k-\xi_k) A_k\eta,
\]
which implies that $G$ is given by \eqref{g-span}.\smallskip

Conversely, assuming $G$ is given by \eqref{g-span} and $f=G\eta$, where $\eta$ is a constant column matrix with entries in $\mathcal{A}$. Then, because
\[
G(\xi) = G(\zeta) \odot_{\scriptscriptstyle \xi} \left(I-\sum_{k=1}^m(\zeta_k-\xi_k)A_k\right),
\]
we have
\begin{eqnarray*}
f(\zeta)-f(\xi) & = & G(\xi)\odot_{\scriptscriptstyle \xi}\left(I-\sum_{k=1}^m(\zeta_k-\xi_k)A_k\right)^{-\odot_{\scriptscriptstyle \xi}}\eta-G(\xi)\eta \\
    & = & G(\zeta)\odot_{\scriptscriptstyle \xi}\left(I-\left(I-\sum_{k=1}^m(\zeta_k-\xi_k)A_k\right)\right)\eta \\
    & = & \sum_{k=1}^m(\zeta_k-\xi_k)\odot_{\scriptscriptstyle \xi} GA_k\eta,
\end{eqnarray*}
i.e., there exist functions $g_k=GA_k\eta\in\mathcal{G}$ which solves Gleason's problem for $f$.
\end{proof}

Now, we show that backward-shift-invariant functions are a particular type of resolvent-invariant functions.

\begin{corollary}
A function $f$ belongs to a finite-dimensional backward-shift invariant space $\mathscr{R}$ if and only it can be spanned by the columns of a matrix-valued function $G$ given by \eqref{g-span} where the matrices $A_k$, $k\in\mathbb{Z}_m$, commute among themselves.
\end{corollary}

\begin{proof}
Because we already know that $g_k=\R_k(\xi)f$ is a solution to Gleason's problem, we can use the fact that $\mathscr{R}\subset\mathcal{G}$. Then, let $f\in\mathscr{R}$ be given by $f=G\eta$, where $G$ is of the form \eqref{g-span}, and $\eta$ is a constant column matrix with entries in $\mathcal{A}$. Hence, our goal is to show that the constant matrices $A_k$, $k\in\mathbb{Z}_m$, in the definition of $G$ are commutative.\smallskip

From the proof of Proposition \ref{resolvent}, we know that $g_k=GA_k\eta$ $\R_k(\xi)f\in\mathcal{G}$. Therefore, we must have $\R_k(\xi)G=GA_k$. Now, because the operators $\R_k(\xi)$ commute, $\R_k(\xi)\R_j(\xi)G=\R_j(\xi)\R_k(\xi)G$ and, then, $A_kA_j=A_jA_k$ for every $k,j\in\mathbb{Z}_m$, as we wanted to show.\smallskip

Conversely, let $f$ be spanned by the columns of $G$ given by \eqref{g-span} with the constant matrices $A_k$ being commutative. Let, moreover, $\eta$ be a constant column matrix with entries in $\mathcal{A}$ such that $f=G\eta$. Then, our goal is to show that the solutions $g_k=GA_k\eta$ to Gleason's problem can be expressed as $g_k=\R_k(\xi)f$.\smallskip

First, observe that the commutativity of the matrices $A_k$ allows us to write $G$ as
\[
G(\zeta) = \sum_{\alpha\in\mathbb{N}_0^m} \frac{|\alpha|!}{\alpha!} (\zeta-\xi)^\alpha G(\xi) A^\alpha,
\]
where $A^\alpha= A_1^{\alpha_1}A_2^{\alpha_2}\cdots A_m^{\alpha_m}$. Then, using Proposition \ref{rk-ps},
\begin{eqnarray*}
(\R_k(\xi) f)(\zeta) & = & \sum_{\substack{\alpha\in\mathbb{N}_0^m; \\ \alpha\ge \iota_k}} \frac{(|\alpha|-1)!}{(\alpha-\iota_k)!} (\zeta-\xi)^{\alpha-\iota_k} G(\xi) A^\alpha\eta \\
              & = & \sum_{\substack{\alpha\in\mathbb{N}_0^m; \\ \alpha\ge \iota_k}} \frac{(|\alpha|-1)!}{(\alpha-\iota_k)!} (\zeta-\xi)^{\alpha-\iota_k} G(\xi) A^{\alpha-\iota_k} A_k\eta \\
              & = & G(\zeta)A_k\eta \\
              & = & g_k(\zeta),
\end{eqnarray*}
as we wanted to show.
\end{proof}

\begin{remark}
Before we go to the next section, observe that if $f\in\mathcal{G}$, there exists a constant column matrix $\eta$ with entries in $\mathcal{A}$ such that $f = G\eta$, where $G$ is given by \eqref{g-span}. Also, $f$ admits solutions to Gleason's problem, which are given by $g_k=GA_k\eta$. Then,
\begin{eqnarray*}
f(\zeta) & = & f(\xi)+G(\zeta)\odot_{\scriptscriptstyle \xi}\sum_{k=1}^m (\zeta_k-\xi_k) A_k\eta \\
         & = & f(\xi) + G(\xi)\odot_{\scriptscriptstyle \xi} \left(I-\sum_{k=1}^m (\zeta_k-\xi_k) A_k\right)^{-\odot_{\scriptscriptstyle \xi}} \odot_{\scriptscriptstyle \xi} \sum_{k=1}^m (\zeta_k-\xi_k) A_k\eta.
\end{eqnarray*}
The above expression characterizes a {\it hyperholomorphic rational function}, the topic of our next section.
\label{g-rational}
\end{remark}

\section{Hyperholomorphic rational functions}
\setcounter{equation}{0}
\label{sec-rat}

Rational functions are, on the one hand, simply quotients of polynomials. On the other hand, they are a fundamental player in many areas of analysis and other related topics. In this section, we study hyperholomorphic rational functions or rational functions of Fueter variables. Our focus is, in particular, the rational functions which are analytic in a neighborhood a fixed $\xi\in\mathscr{H}^m$. From the theory of linear systems, we know that such rational functions can be written in the following form
\begin{equation}
R(\zeta)=D+C\odot_{\scriptscriptstyle \xi} \left(I-\sum_{k=1}^m(\zeta_k-\xi_k) A_k\right)^{-\odot_{\scriptscriptstyle \xi}} \odot_{\scriptscriptstyle \xi} \left(\sum_{k=1}^m(\zeta_k-\xi_k) B_k\right),
\label{real123}
\end{equation}
where $A_k,B_k,C$ and $D$, $k\in\mathbb{Z}_m$, are matrices of appropriate sizes with entries in $\mathcal{A}$. Expression \eqref{real123} is called a {\it realization} of $R$.\smallskip

Observe that expression \eqref{real123} is trivially a ratio, with respect to the $\odot$-product, of Fueter polynomials. The main goal of the next results we present in the sequel is to prove the converse, i.e., to show that every rational function analytic at the origin admits a realization \eqref{real123}. We start by showing that the inverse of a function that is invertible at the origin and whose realization is given by \ref{real123} also admits a realization.

\begin{proposition}
\label{leminv}
Assume that $D$ in \eqref{real123} is invertible. Then,
\begin{equation}
R(\zeta)^{-\odot_{\scriptscriptstyle \xi}}=D^{-1}-D^{-1}C\odot_{\scriptscriptstyle \xi} \left(I-\sum_{k=1}^m(\zeta_k-\xi_k)A_k^{\Box}\right)^{-\odot_{\scriptscriptstyle \xi}} \odot_{\scriptscriptstyle \xi} \left(\sum_{k=1}^m(\zeta_k-\xi_k)B_k\right)D^{-1},
\label{leminv1}
\end{equation}
where
\begin{equation}
A_k^\Box=A_k-B_kD^{-1}C.
\end{equation}
\end{proposition}

\begin{proof}
The proof follows in a similar way to the one presented for the quaternions in \cite{MR2124899} and for the Grassmann algebra in \cite{alpay2018positivity}. Note that
\begin{eqnarray*}
\hspace{-.2cm}R(\zeta)\odot_{\scriptscriptstyle \xi} R(\zeta)^{-1} & \hspace{-.2cm}= & \hspace{-.2cm} I-C \odot_{\scriptscriptstyle \xi} \left[- \left(I-\sum_{k=1}^m(\zeta_k-\xi_k)A_k\right)^{-\odot_{\scriptscriptstyle \xi}}\hspace{-.2cm}\odot_{\scriptscriptstyle \xi} G(\zeta)\odot_{\scriptscriptstyle \xi} \left(I-\sum_{k=1}^m(\zeta_k-\xi_k)A_k^{\Box}\right)^{-\odot_{\scriptscriptstyle \xi}}\hspace{-.2cm}- \right. \\
          & \hspace{-.2cm} & \hspace{-.2cm}- \left. \left(I-\sum_{k=1}^m(\zeta_k-\xi_k)A_k^{\Box}\right)^{-\odot_{\scriptscriptstyle \xi}} + \left(I-\sum_{k=1}^m(\zeta_k-\xi_k)A_k\right)^{-\odot_{\scriptscriptstyle \xi}}  \right] \odot_{\scriptscriptstyle \xi} \left(\sum_{k=1}^m(\zeta_k-\xi_k)B_k\right)D^{-1},
\end{eqnarray*}
where
\begin{eqnarray*}
G(\zeta) & = & \left(\sum_{k=1}^m(\zeta_k-\xi_k)B_k\right)D^{-1}C = \sum_{k=1}^m(\zeta_k-\xi_k)\left(A_k-A_k^\Box\right) \\
         & = & \left(I-\sum_{k=1}^m(\zeta_k-\xi_k)A_k^\Box\right) - \left(I-\sum_{k=1}^m(\zeta_k-\xi_k)A_k\right).
\end{eqnarray*}
Therefore,
\[
R(\zeta)\odot_{\scriptscriptstyle \xi} R(\zeta)^{-1} = I,
\]
as we wanted to show.
\end{proof}

\begin{proposition}
\label{lemsumprod}
Let
\begin{equation}
R_u(\zeta)=D_u+C_u\odot_{\scriptscriptstyle \xi} \left(I_{N_u}-\sum_{k=1}^m(\zeta_k-\xi_k)(A_u)_k\right)^{-\odot_{\scriptscriptstyle \xi}} \odot_{\scriptscriptstyle \xi} \left(\sum_{k=1}^m(\zeta_k-\xi_k)(B_u)_k\right),
\end{equation}
where $u=1,2$, be two realizations of rational functions with compatible sizes. Then, for $k\in\mathbb{Z}_m$,\\
$(1)$ a realization of $R_1(\zeta)\odot_{\scriptscriptstyle \xi} R_2(\zeta)$ is given by
\begin{equation}
A_k=\begin{pmatrix}(A_1)_k&(B_k)_1C_2\\ 0&(A_2)_k\end{pmatrix},\quad B_k=\begin{pmatrix} (B_1)_kD_2\\ (B_2)_k\end{pmatrix},\quad C=\begin{pmatrix} C_1& D_1C_2\end{pmatrix},\quad D=D_1D_2;
\end{equation}
$(2)$ a realization of $R_1(\zeta)+R_2(\zeta)$ is given by
\begin{equation}
A_k=\begin{pmatrix}(A_1)_k&0\\ 0&(A_2)_k\end{pmatrix},\quad B_k=\begin{pmatrix} (B_1)_k\\ (B_2)_k\end{pmatrix},\quad C=\begin{pmatrix} C_1& C_2\end{pmatrix},\quad D=D_1+D_2;
\end{equation}
$(3)$ a realization of $\begin{pmatrix}R_1(\zeta)&R_2(\zeta)\end{pmatrix}$ is given by
\begin{equation}
A_k=\begin{pmatrix}(A_1)_k&0\\ 0&(A_2)_k\end{pmatrix},\quad B_k=\begin{pmatrix} (B_1)_k&0\\0& (B_2)_k\end{pmatrix},\quad C=\begin{pmatrix} C_1& C_2\end{pmatrix},
\quad D=\begin{pmatrix}D_1& D_2\end{pmatrix};
\end{equation}
$(4)$ a realization of $\begin{pmatrix}R_1(\zeta)\\R_2(\zeta)\end{pmatrix}$ is given by
\begin{equation}
A_k=\begin{pmatrix}(A_1)_k&0\\ 0&(A_2)_k\end{pmatrix},\quad B_k=\begin{pmatrix} (B_1)_k\\(B_2)_k\end{pmatrix},\quad C=\begin{pmatrix} C_1&0\\ 0& C_2\end{pmatrix},
\quad D=\begin{pmatrix}D_1\\ D_2\end{pmatrix}.
\end{equation}
\end{proposition}

Similarly to the proof of Proposition \ref{leminv}, the proof of the above proposition follows like the one for the classical case. It is, then, omitted here.\smallskip

With the results presented until now, we already know that the $\odot_{\scriptscriptstyle \xi}$-inverse of a realization and that the $\odot_{\scriptscriptstyle \xi}$-product of the two realizations admit a realization themselves (Proposition \ref{leminv} and the first part of Proposition \ref{lemsumprod}). Therefore, we only need to show that every Fueter polynomial admits a realization to conclude that every rational function with Fueter variables analytic at the origin admits a realization \eqref{real123}.

\begin{proposition}
\label{polreal}
Any Fueter polynomial admits a realization.
\end{proposition}

\begin{proof}
In view of Proposition \ref{lemsumprod}, it suffices to prove that constant terms and terms of the form $(\zeta_k-\xi_k)M$, $k\in\mathbb{Z}_m$, admit realizations. But this is clear. Indeed, a constant matrix $M$ corresponds to the realization $A_j=B=C=0$ and $D=M$, $j\in\mathbb{Z}_m$. Moreover, the function $(\zeta_k-\xi_k)M$ corresponds to $C=M$, $A_j=D=0$, $B_j=\delta_{jk}I_N$, $j\in\mathbb{Z}_m$.
\end{proof}

Thus the following theorem has been proved:
\begin{theorem}
Let $R$ be a function of Fueter variables analytic at the origin. Then, $R$ is rational if and only if it admits a realization given by \eqref{real123}.
\end{theorem}

With the next theorem, we present two more characterizations of rational functions.

\begin{theorem}
Assume $R$ is a function of Fueter variables analytic at the origin. Then, $R$ is rational if and only if\\
$(1)$ its Taylor coefficients are given by
\begin{equation}
r_\alpha=\left\{\begin{array}{l l}
                  D, & if \ |\alpha|=0\\
                  \frac{(|\alpha|-1)!}{\alpha!} C\left(\sum_{k=1}^m \alpha_k A^{\alpha-\iota_k}B_k\right), & if \ |\alpha|\ge 1
\end{array}\right.,
\label{taylor-coeff}
\end{equation}
where $\alpha\in\mathbb{N}_0^m$ and $A^\alpha = A_1^{\alpha_1} \times \cdots \times A_m^{\alpha_m}$; \\

$(2)$ there exists a finite-dimensional resolvent-invariant space $\mathcal{G}$ such that Gleason's problem is solvable for every $f=R\eta\in\mathcal{G}$, where $\eta$ is a constant column matrix with entries in $\mathcal{A}$.
\end{theorem}

\begin{proof}
The proof of $(1)$ follows from direct computation since $R$ is rational if and only if
\begin{eqnarray*}
R(\zeta) & = & D+C\odot_{\scriptscriptstyle \xi} \left(I-\sum_{k=1}^m(\zeta_k-\xi_k) A_k\right)^{-\odot_{\scriptscriptstyle \xi}} \odot_{\scriptscriptstyle \xi} \left(\sum_{k=1}^m(\zeta_k-\xi_k) B_k\right) \\
         & = & D+\left(\sum_{\alpha\in\mathbb{N}_0^m}\frac{|\alpha|!}{\alpha!}(\zeta-\xi)^\alpha CA^\alpha\right)\odot_{\scriptscriptstyle \xi} \left(\sum_{k=1}^m(\zeta_k-\xi_k) B_k\right) \\
         & = & D+\sum_{k=1}^m\sum_{\substack{\alpha\in\mathbb{N}_0^m;\\ \alpha\ge\iota_k}}\frac{(|\alpha|-1)!}{\alpha!}(\zeta-\xi)^\alpha C \alpha_k A^{\alpha-\iota_k}B_k\\
         & = & \sum_{\alpha\in\mathbb{N}_0^m} (\zeta-\xi)^\alpha r_\alpha,
\end{eqnarray*}
where $r_\alpha$ is given by \eqref{taylor-coeff}.\smallskip

The proof for one direction of $(2)$ is already given in Remark \ref{g-rational}, where it was shown that every $f$ in a finite-dimensional $\mathcal{G}$ admits a realization. For the converse, we assume $R$ is rational, and there exists $\eta$ such that $f=R\eta$. Then,
\[
f(\zeta)-f(\xi) = C\odot_{\scriptscriptstyle \xi} \left(I-\sum_{k=1}^m(\zeta_k-\xi_k) A_k\right)^{-\odot_{\scriptscriptstyle \xi}}\odot_{\scriptscriptstyle \xi}\left(\sum_{k=1}^m(\zeta_k-\xi_k) B_k\right)\eta
\]
and, moreover,
\[
g_k = C\odot_{\scriptscriptstyle \xi} \left(I-\sum_{k=1}^m(\zeta_k-\xi_k) A_k\right)^{-\odot_{\scriptscriptstyle \xi}} B_k\eta
\]
solves Gleason's problem in a space $\mathcal{G}$ generated by the columns of a function of the type \eqref{g-span}.
\end{proof}

\section{Banach modules of Fueter series}
\setcounter{equation}{0}
\label{sec-banach}

Let $\mathbf c=(c_\alpha)_{\alpha\in\mathfrak{I}_0}$ be a family of non-null real numbers. The case where some of the coefficients $c_\alpha$ are zero is easily adapted.\smallskip

We, then, set
\begin{equation}
K_{\mathbf c}(\zeta,\xi)=\sum_{\alpha\in\mathbb N_0^m}\frac{\zeta^\alpha (\xi^\alpha)^\dagger}{c_\alpha},
\label{kernel}
\end{equation}
assuming that the set
\[
\Omega(\mathbf c)=\left\{\zeta\in\mathscr{H}^m \suchthat \sum_{\alpha\in\mathbb N_0^m} \frac{(N(\zeta))^{2\alpha}} {|c_\alpha|}<\infty\right\}
\]
is an open neighborhood of the origin in $\mathscr{H}^m$.\smallskip

Let $\stackrel{\circ}{\mathcal W}\hspace{-1mm}(\mathbf c)$ denote the module of functions $f(\zeta)=\sum_{\alpha\in\mathbb N_0^m}\zeta^\alpha f_\alpha$ with coefficients $f_\alpha\in\mathcal A$ such that
\begin{equation}
\label{ineq123456}
\|f\|\equiv\left(\sum_{\alpha\in\mathbb N_0^m}|c_\alpha| (N(f_\alpha))^2\right)^{1/2}<\infty.
\end{equation}

\begin{remark}
For convenience, we only consider power series centered at the origin hereby. However, one can easily reproduce the results presented in this and in the following sections for power series centered at a different point. Because of this choice, we simply write $\odot$ instead of $\odot_{\scriptscriptstyle 0}$, and $\mathcal{R}_k$ instead of $\mathcal{R}_k(0)$.
\end{remark}

\begin{proposition}
The formula \eqref{ineq123456} defines a norm in $\stackrel{\circ}{\mathcal W}\hspace{-1mm}(\mathbf c)$.
\end{proposition}

\begin{proof}
First, observe that for every $a\in\mathcal{A}$,
\begin{equation}
\|af\| = \left(\sum_{\alpha\in\mathbb N_0^m}|c_\alpha| (N(af_\alpha))^2\right)^{1/2} \le N(a) \|f\|.
\end{equation}
However, if we restrict ourselves to $k\in\mathbb{K}\subset\mathcal{A}$,
\[
\|kf\| = |k| \|f\|,
\]
where we have used \eqref{basic-norm-prop}. Moreover,
\begin{eqnarray*}
\|f\| = 0 & \Rightarrow & \left(\sum_{\alpha\in\mathbb N_0^m}|c_\alpha| (N(f_\alpha))^2\right)^{1/2} = 0 \\
          & \Rightarrow & f_\alpha = 0, \forall \alpha\in\mathbb{N}_0^m \\
          & \Rightarrow & f=0.
\end{eqnarray*}
Finally, if $g(\zeta)=\sum_{\alpha\in\mathbb N_0^m}\zeta^\alpha g_\alpha$ belongs to $\stackrel{\circ}{\mathcal{W}}\hspace{-1mm}(\mathbf{c})$,
\begin{eqnarray*}
\|f+g\|^2 & = & \sum_{\alpha\in\mathbb N_0^m}|c_\alpha| (N(f_\alpha+g_\alpha))^2 \\
          & \le & \sum_{\alpha\in\mathbb N_0^m}|c_\alpha| \left[N(f_\alpha)+N(g_\alpha)\right]^2 \\
          & \le & \sum_{\alpha\in\mathbb N_0^m}|c_\alpha| \left[N(f_\alpha)^2+2N(f_\alpha)N(g_\alpha)+N(g_\alpha)^2\right] \\
          & \le & \|f\|^2 + 2\|f\|\cdot\|g\| + \|g\|^2 \\
          & \le & \left(\|f\|+\|g\|\right)^2.
\end{eqnarray*}
\end{proof}

The next result is to be compared, for instance, to the one presented for the split-quaternions in \cite{alss_IJM}.

\begin{proposition}
Let $f$ and $g$ be two elements of $\stackrel{\circ}{\mathcal{W}}\hspace{-1mm}(\mathbf c)$ -- with $g(\zeta)=\sum_{\alpha\in\mathbb N_0^m}\zeta^\alpha g_\alpha$.
Then, the Hermitian form
\begin{equation}
\label{innerprod}
[f,g]\equiv\sum_{\alpha\in\mathbb N_0^m}c_\alpha g_\alpha^\dagger f_\alpha
\end{equation}
converges in $\mathcal{A}$ and we have
\begin{equation}
N([f,g]) \le \|f\|\cdot\|g\|.
\end{equation}
\end{proposition}

\begin{proof}
The proof follows directly from Cauchy-Schwartz:
\begin{eqnarray*}
N([f,g]) & \le & \sum_{\alpha\in\mathbb N_0^m}|c_\alpha| N(f_\alpha) N(g_\alpha) \\
         & \le & \left(\sum_{\alpha\in\mathbb N_0^m}|c_\alpha| N(f_\alpha)^2\right)^{1/2} \left(\sum_{\alpha\in\mathbb N_0^m}|c_\alpha| N(g_\alpha)^2\right)^{1/2} \\
         & \le & \|f\|\cdot\|g\|
\end{eqnarray*}
\end{proof}

\begin{remark}
If $\mathcal{A}$ does not have zero divisors, the results presented by Paschke in \cite{wp} can be applied to our study since the form satisfies the conditions to be what is defined as an $\mathcal{A}$-valued inner product. We, however, want to study a more generic scenario in this work and allow $\mathcal{A}$ to have zero divisors.
\end{remark}

\begin{proposition}
With the Hermitian form \eqref{innerprod}, $\stackrel{\circ}{\mathcal{W}}\hspace{-1mm}(\mathbf c)$ admits the reproducing kernel $K_\mathbf{c}$ given by \eqref{kernel}, i.e.,
\begin{equation}
\label{innerprod1}
[f(\cdot),K_{\mathbf c}(\cdot,\xi)b]=b^\dag f(\xi)
\end{equation}
for every $b\in\mathcal{A}$.
\end{proposition}

\begin{proof}
By definition,
\begin{eqnarray*}
[f(\cdot),K_{\mathbf c}(\cdot,\xi)b] & = & \sum_{\alpha\in\mathbb N_0^m}c_\alpha \left(b^\dag \frac{\xi^\alpha}{c_\alpha}\right) f_\alpha \\
             & = & b^\dag \sum_{\alpha\in\mathbb N_0^m} \xi^\alpha f_\alpha \\
             & = & b^\dag f(\xi).
\end{eqnarray*}
\end{proof}

{\it A priori}, the module $\stackrel{\circ}{\mathcal W}\hspace{-1mm}(\mathbf c)$ is not complete. By a general theorem on metric spaces, it has a completion to a Banach space, which is unique up to a metric space (Banach space?) isometry. In view of \eqref{innerprod1}, we now prove that $\stackrel{\circ}{\mathcal{W}}\hspace{-1mm}(\mathbf{c})$ has a uniquely defined completion which is a module of functions. The argument follows 
a known argument in the theory of reproducing kernel Hilbert spaces. As in that case, we consider the vector module of Cauchy sequences in $\stackrel{\circ}{\mathcal W}\hspace{-1mm}(\mathbf c)$, and say that two such sequences $(f_n)$ and $(g_n)$ are equivalent if
\[
\lim_{n\rightarrow \infty}\|f_n-g_n\|=0
\]
This is indeed an equivalence relation, and we denote by $\mathcal{CS}$ the quotient module. Still like the classical case, the formula
\begin{equation}
\label{norm12345}
\|\stackrel{\sim}{f}\|=\lim_{n\rightarrow\infty}\|f_n\|
\end{equation}
where $(f_n)$ belongs to the equivalence class $\stackrel{\sim}{f}$ does not depend on the given representative in the equivalence class, and defines a norm on $\mathcal{CS}$.
It now follows from \eqref{innerprod} and \eqref{innerprod1} that for $f\in\mathcal{CS}$ the limit
\begin{equation}
\label{lim345}
b^\dag f(w)\equiv\lim_{n\rightarrow\infty} b^\dag f_n(w)
\end{equation}
exists in the topology of $\mathcal A$ and does not depend on the given representative of the equivalence class. \smallskip

We denote by $\mathcal{W}(\mathbf{c})$ the representation of $\mathcal{CS}$ as a module of functions. We, then, have the following result:

\begin{theorem}
The space $\mathcal{W}(\mathbf{c})$ endowed with the norm \eqref{norm12345} is a Banach module, in which $\stackrel{\circ}{\mathcal{W}}\hspace{-1mm}(\mathbf{c})$ is naturally embedded in a dense way. Moreover, \eqref{innerprod1} holds in $\mathcal{W}(\mathbf{c})$.
\end{theorem}

\begin{remark}
If for certain algebra the form \eqref{innerprod} is such that $\left[f,f\right]\geq0$, i.e., it is a positive real number for every $f$ in $\mathcal{W}(\mathbf{c})$, then $\mathcal{W}(\mathbf{c})$ is a Hilbert module. If on the other hand $\left[f,f\right]$ is a real number for every $f\in\mathcal{W}(\mathbf{c})$ but it is also negative to some $f$, then $\mathcal{W}(\mathbf{c})$ is a Potryagin or a Krein module. For example, the Banach module of power series associated with modules with positive coefficients $c_\alpha$ when the algebra $\mathcal{A}$ is the quaternionic algebra is a Hilbert space; Pontryagin or Krein modules, on the other hand, are what appear in the case of the split quaternions. For information on Pontryagin and Krein spaces; see, e.g., \cite{azih,bognar,MR92m:47068}.
\end{remark}

Now, we want to start studying operators in $\mathcal{W}(\mathbf{c})$ and their adjoint with respect to the Hermitian form \eqref{innerprod}. First, we present the following preliminary result.

\begin{proposition}
Let $g\in\mathcal{W}(\mathbf{c})$ such that
\begin{equation}
\left[f,g\right] = 0
\label{0-element}
\end{equation}
for every $f\in\mathcal{W}(\mathbf{c})$. Then, $g=0$. Analogously, if for a fixed $g\in\mathcal{W}(\mathbf{c})$ expression \eqref{0-element} holds for every $f\in\mathcal{W}(\mathbf{c})$, then $f=0.$
\label{form-result}
\end{proposition}

\begin{proof}
Note that for every $\alpha\in\mathbb{N}_0^m$
\[
\left[\zeta^\alpha,g\right] = c_\alpha g_\alpha^\dag.
\]
Because $c_\alpha\in\mathbb{R}$ is a positive number, expression \eqref{0-element} implies that $g_\alpha = 0$. Then, $g=0$.\smallskip

This also proves the analogous result. Just observe that
\[
\left[f,g\right] = \left(\left[g,f\right]\right)^\dag.
\]
\end{proof}

The uniqueness of the adjoint of operators $O$ in $\mathcal{W}(\mathbf{c})$ follows directly form the above result, as we present next.

\begin{proposition}
Let $O$ be an operator in $\mathcal{W}(\mathbf{c})$ and assume that it admits an adjoint, i.e., there exists $A$ which is characterized by
\[
\left[Af,g\right] \equiv \left[f,Og\right]
\]
for every $f,g\in\mathcal{W}(\mathbf{c})$ for which the term on the right-hand side converges. Then, $A$ is unique, and we denote $A=O^*$.
\label{uniqueness}
\end{proposition}

\begin{proof}
Suppose $O$ admits two adjoints, say $A_1$ and $A_2$. Then, note that for every $f,g\in\mathcal{W}(\mathbf{c})$
\[
\left[A_1f-A_2f,g\right] = \left[A_1f,g\right] - \left[A_2f,g\right] = \left[f,Og\right] - \left[f,Og\right] = 0.
\]
Therefore, by Proposition \ref{form-result}, we conclude that
\[
A_1f-A_2f = 0
\]
for every $f\in\mathcal{W}(\mathbf{c})$, which implies that $A_1=A_2$.
\end{proof}

Now, we present at least one condition for the existence of the adjoint of a certain operator $O$. To do so, we first need to introduce some definitions and to explore more the structure of the module $\mathcal{W}(\mathbf{c})$.\smallskip

First, observe that
\[
\left[f,g\right] = \sum_{\alpha\in\mathbb{N}_0^m} c_\alpha g_\alpha^\dag f_\alpha = \sum_{\ell=0}^m \left(\sum_{\alpha\in\mathbb{N}_0^m}c_\alpha g_\alpha^\dag\chi_\ell f_\alpha\right)e_\ell,
\]
where we have used the characteristic operators of the algebra $\mathcal{A}$ defined in \eqref{chi-def}. We, then, identify $\mathbb{K}$-valued inner product structures in certain spaces -- denoted by $\mathcal{K}_\ell(\mathbf{c})$ -- and define
\begin{equation}
\left\langle f,g\right\rangle_{\mathcal{K}_\ell(\mathbf{c})} \equiv \sum_{\alpha\in\mathbb{N}_0^m}c_\alpha g_\alpha^\dag\chi_\ell f_\alpha.
\label{k-inner-product}
\end{equation}
Observe that each $\mathcal{K}_\ell(\mathbf{c})$ constitutes, in general, a Krein space, which includes, in particular, Pontryagin and Hilbert spaces. 
What determines if it ends up being a Hilbert space or a Krein space is the coefficients $c_\alpha$ together with the characteristic operator $\chi_\ell$.\smallskip

Now, let us define a contractive operator. First, we introduce the notion of positivity in $\mathcal{A}$. Such a notion can be naturally given by stating that the quadratic form $aa^\dag$ is non-negative for every $a\in\mathcal{A}$ and writing
\[
aa^\dag \succeq 0.
\]
Such a definition does not always imply real positivity even if $aa^\dag$ is a real number. For example, it can be a negative real number in the setting of the split-quaternions.\smallskip

\begin{definition}
An operator $O$ in $\mathcal{W}(\mathbf{c})$ is said to be a contraction if
\[
\left[Of,Of\right] \preceq \left[f,f\right]
\]
for every $f\in\mathcal{W}(\mathbf{c})$.
\label{def-cont-op}
\end{definition}

\begin{proposition}
If the operator $O$ is bounded in every $\mathcal{K}_\ell(\mathbf{c})$, then it admits an adjoint in $\mathcal{W}(\mathbf{c})$.
\end{proposition}

\begin{proof}
Since $O$ is bounded in every $\mathcal{K}_\ell(\mathbf{c})$, then in each of such spaces it admits an adjoint, i.e., there exists an operator $A_\ell$ such that
\begin{equation}
\left\langle Of,g\right\rangle_{\mathcal{K}_\ell(\mathbf{c})} = \left\langle f,A_\ell g\right\rangle_{\mathcal{K}_\ell(\mathbf{c})}.
\label{adj-rel}
\end{equation}
Let us now show that the existence of such an adjoint in every $\mathcal{K}_\ell(\mathbf{c})$ implies the existence of an adjoint in $\mathcal{W}(\mathbf{c})$. First, observe that
\[
\left\langle Of,g\right\rangle_{\mathcal{K}_\ell(\mathbf{c})} = \left[\chi_\ell Of,g\right]
\]
and
\[
\left\langle f,A_\ell g\right\rangle_{\mathcal{K}_\ell(\mathbf{c})} = \left[\chi_\ell f,A_\ell g\right].
\]
The above two expressions together with \eqref{adj-rel} lead to
\[
\left[\chi_\ell Of,g\right] = \left[f,\chi_\ell^*A_\ell g\right],
\]
which, in turn, implies that
\[
\left[Of,g\right] = \left[f,\left(\sum_{k=0}^m e_\ell^\dag\chi_\ell^*A_\ell\right) g\right].
\]
Then, using the definition of the adjoint operator, we conclude that
\[
O^* = \sum_{k=0}^m e_\ell^\dag\chi_\ell^*A_\ell.
\]
\end{proof}

An operator that plays an important role in spaces of power series is the multiplication operator. We, then, define the analogous of this operator in our setting.\smallskip

Let $\mathcal{M}_{\zeta_k}$, $k\in\mathbb{Z}_m$, denote the $\odot$-multiplication operator by $\zeta_k$, i.e., if $f$ belongs to $\mathcal{W}(\mathbf{c})$, then
\begin{equation}
\mathcal{M}_{\zeta_k}f = \zeta_k \odot f.
\label{multiplication}
\end{equation}
Such an operator is further explored in particular examples of Banach modules of Fueter series in the next sections.\smallskip

Before, we generalize it in the following way: let $s$ be a Fueter series which belongs to some Banach module of Fueter series. Moreover, let $\mathcal{W}(\mathbf{c})$ and $\mathcal{W}(\mathbf{d})$ be two Banach modules of Fueter series. Then, the multiplier $\mathcal{M}_s: \mathcal{W}(\mathbf{c}) \rightarrow \mathcal{W}(\mathbf{d})$ is defined as
\[
\mathcal{M}_s f = s \odot f,
\]
where $f\in\mathcal{W}(\mathbf{c})$ and $g=s\odot f\in\mathcal{W}(\mathbf{d})$.

\begin{proposition}
The following holds in $\mathcal{W}(\mathbf{c})$:
\[
\left[\mathcal{M}^*_s K_\mathbf{d}(\cdot,\xi)b_1,K_\mathbf{c}(\cdot,\zeta)b_2\right]_{\mathcal{W}(\mathbf{c})} = b_2^\dag \left(\sum_{\alpha\in\mathbb{N}_0^m} \frac{(\zeta^\alpha \odot s(\zeta))\xi^\dag}{c_\alpha}\right) b_1,
\]
where $b_1,b_2\in\mathcal{A}$.
\end{proposition}

\begin{proof}
This follows from a simple computation since
\begin{eqnarray*}
\left[\mathcal{M}^*_s K_\mathbf{d}(\cdot,\xi)b_1,K_\mathbf{c}(\cdot,\zeta)b_2\right]_{\mathcal{W}(\mathbf{c})} & = & \left[K_\mathbf{d}(\cdot,\xi)b_1,\mathcal{M}_s K_\mathbf{c}(\cdot,\zeta)b_2\right]_{\mathcal{W}(\mathbf{d})}\\
              & = & \left(\left[s\odot K_\mathbf{c}(\cdot,\zeta)b_2,K_\mathbf{d}(\cdot,\xi)b_1\right]_{\mathcal{W}(\mathbf{d})}\right)^\dag\\
              & = & \left(b_1^\dag s\odot K_\mathbf{c}(\xi,\zeta)b_2\right)^\dag\\
              & = & b_2^\dag \left(\sum_{\alpha\in\mathbb{N}_0^m} \frac{\xi(\zeta^\alpha \odot s(\zeta))^\dag}{c_\alpha}\right)^\dag b_1\\
              & = & b_2^\dag \left(\sum_{\alpha\in\mathbb{N}_0^m} \frac{(\zeta^\alpha \odot s(\zeta))\xi^\dag}{c_\alpha}\right) b_1.
\end{eqnarray*}
\end{proof}

\section{Drury-Arveson module}
\setcounter{equation}{0}
\label{sec-arveson}

In this section, we will study a particular case of a reproducing kernel module which has its coefficients $c_\alpha$ given by
\[
c_\alpha = \frac{\alpha!}{|\alpha|!},
\]
i.e., the module $\mathcal{W}(\mathbf{c})$ with reproducing kernel
\begin{equation}
K_\mathbf{c}(\zeta,\xi) = \sum_{\alpha\in\mathbb{N}_0^m}\frac{|\alpha|!}{\alpha!} \zeta^\alpha\left(\xi^\alpha\right)^\dag
\label{arveson-rk}
\end{equation}
and Hermitian form
\[
\left[f,g\right] = \sum_{\alpha\in\mathbb{N}_0^m} \frac{\alpha!}{|\alpha|!} g_\alpha^\dag f_\alpha,
\]
which is called the Drury-Arveson module.\smallskip

\begin{proposition}
For every $k\in\mathbb{Z}_m$, the operator $\mathcal{M}_{\zeta_k}$ is a contraction from $\mathcal{W}(\mathbf{c})$ into itself.
\end{proposition}

\begin{proof}
The proof follows from a direct computation. For every $f\in\mathcal{W}(\mathbf{c})$, it holds that
\begin{eqnarray*}
\left[\mathcal{M}_{\zeta_k}f,\mathcal{M}_{\zeta_k}f\right] & = & \sum_{\alpha\in\mathbb{N}_0^m} \left[ \zeta^{\alpha+\iota_k} f_\alpha, \zeta^{\alpha+\iota_k} f_\alpha \right] \\
             & = & \sum_{\alpha\in\mathbb{N}_0^m} \frac{(\alpha+\iota_k)!}{(|\alpha|+1)!} f_\alpha^\dag f_\alpha \\
             & = & \sum_{\alpha\in\mathbb{N}_0^m} \frac{\alpha_k+1}{|\alpha|+1}\left( c_\alpha f_\alpha^\dag f_\alpha\right).
\end{eqnarray*}

Observe that
\[
\frac{|\alpha|-\alpha_k}{|\alpha|+1} \geq 0
\]
for every $\alpha\in\mathbb{N}_0^m$ and, then, it admits a non-negative square root. Therefore, writing
\[
d_\alpha = \sqrt{\frac{c_\alpha(|\alpha|-\alpha_k)}{|\alpha|+1}},
\]
we conclude that
\begin{eqnarray*}
\left[f,f\right] - \left[\mathcal{M}_{\zeta_k}f,\mathcal{M}_{\zeta_k}f\right] & = & \sum_{\alpha\in\mathbb{N}_0^m} \frac{|\alpha|-\alpha_k}{|\alpha|+1} c_\alpha f_\alpha^\dag f_\alpha \\
             & = & \sum_{\alpha\in\mathbb{N}_0^m} \left(d_\alpha f_\alpha\right)^\dag \left(d_\alpha f_\alpha\right) \\
             & \succeq & 0,
\end{eqnarray*}
as we wanted to show.
\end{proof}

\begin{proposition}
For every $k\in\mathbb{Z}_m$, the adjoint of $\mathcal{M}_{\zeta_k}$ is the backward-shift operator $\R_k$, i.e.,
\[
\left[\R_kf,g\right] = \left[f,\mathcal{M}_{\zeta_k}g\right]
\]
for every $f,g\in\mathcal{W}(\mathbf{c})$.
\end{proposition}

\begin{proof}
The proof is similar to the one for the quaternions seen in \cite{MR2124899}. First, observe that $\forall\beta\in\mathbb{N}_0^m$ and $\alpha\in\mathbb{N}_0^m$ such that $\alpha\ge\iota_k$, $k\in\mathbb{Z}_m$, we have
\[
\left[\R_k\zeta^\alpha,\zeta^\beta\right] = \left[\frac{\alpha_k}{|\alpha|}\zeta^{\alpha-\iota_k},\zeta^\beta\right] = \frac{(\beta+\iota_k)!}{(|\beta|+1)!} \delta_{\alpha-\iota_k,\beta} = \left[\zeta^\alpha,\zeta^{\beta+\iota_k}\right].
\]
Finally, if $\beta$ is such that $\beta_k=0$,
\[
\left[\R_k\zeta^\alpha,\zeta^\beta\right] = 0 = \left[\zeta^\alpha,\zeta^{\beta+\iota_k}\right] = \left[\zeta^\alpha,\mathcal{M}_{\zeta_k}\zeta^\beta\right].
\]
\end{proof}

Already preparing for the next section, let $C$ be the operator of evaluation at the origin, i.e., $Cf=f(0)$ for every $f\in\mathcal{W}(\mathbf{c})$. Then, the identity
\begin{equation}
I-\sum_{k\in\mathbb{Z}_m}\mathcal{M}_{\zeta_k} \mathcal{M}_{\zeta_k}^* = C^*C
\label{identity}
\end{equation}
holds.

\section{Blaschke factor}
\setcounter{equation}{0}
\label{sec-blaschke}

We denoted the transpose of $\xi = (\xi_1 \quad \xi_2 \quad \cdots \quad \xi_m)$ by
\[
\xi^* = \left(
\begin{array}{c}
\xi_1^\dag \\
\xi_2^\dag \\
\vdots \\
\xi_m^\dag
\end{array}
\right).
\]
For every $\xi$ such that $\left\|\xi\right\|_{\mathcal{A}^m} <1$, where the norm in $\mathcal{A}^m$ was defined in \eqref{norm-am} and can be written as
\[
\left\|\xi\right\|_{\mathcal{A}^m} = \left(\sum_{k=1}^m N(\xi_k)^2\right)^{1/2},
\]
we define the Blaschke factor $B_\xi$ as
\[
B_\xi = (1-\xi\xi^*)^{1/2}\odot(1-\zeta\xi^*)^{-\odot} \odot (\zeta-\xi)(I-\xi^*\xi)^{-1/2}.
\]

\begin{proposition}
The following identity holds in the Drury-Arveson module $\mathcal{W}(\mathbf{c})$:
\[
I-\mathcal{B}_\xi \mathcal{B}_\xi^* = (1-\xi\xi^*)^{1/2}\left(I-\mathcal{M}_\zeta \mathcal{M}_\xi^*\right)^{-1} C^*C \left(I-\mathcal{M}_\zeta \mathcal{M}_\xi^*\right)^{-*}(1-\xi\xi^*)^{1/2}.
\]
\end{proposition}

\begin{proof}
The proof follows the one presented in \cite{akap1} and \cite{MR2124899}. First, observe that, because $\left\|\xi\right\|_{\mathcal{A}^m}<1$, the operators $I-\mathcal{M}_\xi\mathcal{M}_\xi^*$ and $I-\mathcal{M}_\xi^*\mathcal{M}_\xi$ are self-adjoint and contractive, their square root -- and the inverse of their square root -- exist.\smallskip

Defining the Halmos extension of $-\mathcal{M}_\xi$ according to \cite{Dym_CBMS}
\[
\mathcal{H} \equiv \left(
\begin{array}{c c}
\left(I-\mathcal{M}_\xi\mathcal{M}_\xi^*\right)^{-1/2} & -\mathcal{M}_\xi\left(I-\mathcal{M}_\xi^*\mathcal{M}_\xi\right)^{-1/2} \\
-\mathcal{M}_\xi^*\left(I-\mathcal{M}_\xi\mathcal{M}_\xi^*\right)^{-1/2} & \left(I-\mathcal{M}_\xi^*\mathcal{M}_\xi\right)^{-1/2}
\end{array}
\right)
\]
and setting
\[
J = \left(
\begin{array}{c c}
I_{\mathcal{W}(\mathbf{c})} & 0 \\
0 & -I_{\mathcal{W}(\mathbf{c})^m}
\end{array}
\right),
\]
the following holds:
\[
\mathcal{H}J\mathcal{H}^* = \mathcal{H}^*J\mathcal{H} = J.
\]
Then, using \eqref{identity},
\begin{eqnarray*}
C^*C & = & I - \mathcal{M}_\zeta\mathcal{M}_\zeta^* \\
     & = & \left(I \quad \mathcal{M}_\zeta\right) J \left(\begin{array}{c} I \\ \mathcal{M}_\zeta\end{array}\right) \\
     & = & \left(I \quad \mathcal{M}_\zeta\right) \mathcal{H}J\mathcal{H}^* \left(\begin{array}{c} I \\ \mathcal{M}_\zeta\end{array}\right) \\
     & = & \left(\mathcal{X}_1 \quad \mathcal{X}_2\right) J \left(\begin{array}{c} \mathcal{X}_1^* \\ \mathcal{X}_2^*\end{array}\right) \\
     & = & \mathcal{X}_1 \mathcal{X}_1^* - \mathcal{X}_2 \mathcal{X}_2^*,
\end{eqnarray*}
where
\[
\mathcal{X}_1 = \left(I-\mathcal{M}_\zeta\mathcal{M}_\xi^*\right) \left(I-\mathcal{M}_\xi\mathcal{M}_\xi^*\right)^{-1/2}
\]
and
\[
\mathcal{X}_2 = \left(\mathcal{M}_\zeta - \mathcal{M}_\xi\right) \left(I-\mathcal{M}_\xi^*\mathcal{M}_\xi\right)^{-1/2}.
\]

The proposition follows directly from the above expression.
\end{proof}

\section{Fock module}
\setcounter{equation}{0}
\label{sec-fock}

In this section, we study the Fock module, which is the particular module of power series $\mathcal{W}(\mathbf{c})$ for which $c_\alpha=\alpha!$. Therefore, its reproducing kernel is
\[
K_\mathbf{c}(\zeta,\xi) = \sum_{\alpha\in\mathbb{N}_0^m} \frac{1}{\alpha!} \zeta^\alpha (\xi^\alpha)^\dag
\]
and it is endowed with the Hermitian form
\[
\left[f,g\right] = \sum_{\alpha\in\mathbb{N}_0^m} \alpha! \ g_\alpha^\dag f_\alpha.
\]

Before presenting a result on the $\odot$-multiplication operator, we define the derivative operator $\partial_k$ in the following way
\[
\partial_k f = \sum_{\alpha\in\mathbb{N}_0^m}\alpha_k \zeta^{\alpha-\iota_k} f_\alpha.
\]
With that, the following result, which is analogous of a result in the classical Fock space, holds in our generalized setting.

\begin{proposition}
For every $k\in\mathbb{Z}_m$, the adjoint of $\mathcal{M}_{\zeta_k}$ in the Fock module $\mathcal{W}(\mathbf{c})$ is the derivative operator, i.e.,
\[
\left[\partial_k f, g\right] = \left[f, \mathcal{M}_{\zeta_k}g\right]
\]
for every $f,g\in\mathcal{W}(\mathbf{c})$.
\end{proposition}

\begin{proof}
This follows from a direct computation. Let $\alpha,\beta\in\mathbb{N}_0^m$. Then,
\begin{eqnarray*}
\left[\partial_k \zeta^\alpha,\zeta^\beta\right] & = & \left[\alpha_k \zeta^{\alpha-\iota_k},\zeta^\beta\right] \\
             & = & (\alpha-\iota_k)! \ \alpha_k \ \delta_{\alpha-\iota_k,\beta} \\
             & = & \alpha! \ \delta_{\alpha,\beta+\iota_k} \\
             & = & \left[\zeta^\alpha,\zeta^{\beta+\iota_k}\right].
\end{eqnarray*}
\end{proof}

\section*{Acknowledgments}
Daniel Alpay thanks the Foster G. and Mary McGaw Professorship in Mathematical Sciences, which supported this research. Ismael L. Paiva acknowledges financial support from the Science without Borders program (CNPq/Brazil). Daniele C. Struppa thanks the Donald Bren Distinguished Chair in Mathematics, which supported this research.

\bibliographystyle{plain}
\bibliography{all}
\end{document}